\documentclass[10pt]{amsart}

\usepackage{amsmath}
\usepackage{amssymb}
\usepackage{amsfonts}%
\usepackage{graphicx} 

\newtheorem{theorem}{Theorem}[section]
\newtheorem{lemma}{Lemma}[section]
\theoremstyle{definition}

\theoremstyle{remark}
\newtheorem{remark}{Remark}[section]
\numberwithin{equation}{section}
\theoremstyle{plain}

\newtheorem{proposition}{Proposition}[section]
\newtheorem{example}{Example}[section]

\begin{document}
\title[Local Limit Theorems for Hitting- and Return Times]{\textbf{Local Limit Theorems for Hitting Times and Return Times of small sets}}
\author{Max Auer}
\address{Department of Mathematics, University of Maryland, College Park, MD 20742-401,
USA }
\email{mauer96@umd.edu}
\author{Roland Zweim\"{u}ller}
\address{Fakult\"{a}t f\"{u}r Mathematik, Universit\"{a}t Wien, Oskar-Morgenstern-Platz
1, 1090 Vienna, Austria }
\email{roland.zweimueller@univie.ac.at}
\urladdr{http://www.mat.univie.ac.at/\symbol{126}zweimueller/ }

\begin{abstract}
We establish abstract local limit theorems for hitting times and return times
of suitable sequences $(A_{l})$ of asymptotically rare events in ergodic
probability preserving dynamical systems, including versions for tuples of
consecutive times and positions of the hits. These results are shown to apply
in the setup of Gibbs-Markov systems. \newline\newline


\end{abstract}
\maketitle

\section{Introduction}%

%TCIMACRO{\TeXButton{NI}{\noindent}}%
%BeginExpansion
\noindent
%EndExpansion
\textbf{Setup and notation.} Let $(X,\mathcal{A},\mu)$ be a probability space,
and $T:X\rightarrow X$ an ergodic $\mu$-preserving map. Take any
$A\in\mathcal{A}$ with $\mu(A)>0$. By ergodicity and the Poincar\'{e}
recurrence theorem, the measurable \emph{(first) hitting time} function of $A
$, $\varphi_{A}:X\rightarrow\overline{\mathbb{N}}:=\{1,2,\ldots,\infty\}$ with
$\varphi_{A}(x):=\inf\{n\geq1:T^{n}x\in A\}$, is finite a.e. on $X$. When
restricted to $A$ it is called the \emph{(first) return time} of that set. Let
$T_{A}x:=T^{\varphi_{A}(x)}x$ for a.e. $x\in X$, which gives the \emph{first
entrance map} $T_{A}:X\rightarrow A$. Its restriction to $A$, the \emph{first
return map} $T_{A}:A\rightarrow A$, is an ergodic measure preserving map on
the probability space $(A,\mathcal{A}\cap A,\mu_{A})$, where $\mu_{A}%
(B):=\mu(A\cap B)/\mu(A)$, $B\in\mathcal{A}$. By Kac' formula, $\int
_{A}\varphi_{A}\,d\mu_{A}=1/\mu(A)$. We will often normalize these functions
accordingly, thus considering $\mu(A)\,\varphi_{A}$.

Our focus will be on hitting times and return times in the limit of very small
sets. Call $(A_{l})_{l\geq1}$ a \emph{sequence of asymptotically rare events}
provided that $A_{l}\in\mathcal{A}$ and $0<\mu(A_{l})\rightarrow0$. There is
now a well-developed theory of distributional convergence of the variables
$\varphi_{A_{l}}$ in such a situation. Here it is important to distinguish
between \emph{convergence in law of (normalized) hitting times},
\begin{equation}
\mu(\mu(A_{l})\varphi_{A_{l}}\leq t)\Longrightarrow F(t)\text{ \quad as
}l\rightarrow\infty\text{,}\label{Eq_IntroHTS}%
\end{equation}
and \emph{convergence in law of (normalized) return times},%
\begin{equation}
\mu_{A_{l}}(\mu(A_{l})\varphi_{A_{l}}\leq t)\Longrightarrow\widetilde
{F}(t)\text{ \quad as }l\rightarrow\infty\text{,}\label{Eq_IntroRTS}%
\end{equation}
the difference being that the functions $\mu(A_{l})\varphi_{A_{l}}$ are viewed
as random variables on one and the same probability space $(X,\mathcal{A}%
,\mu)$ in the first case, but are regarded as variables on the spaces
$(X,\mathcal{A}\cap A_{l},\mu_{A_{l}})$ in the second instance. In
(\ref{Eq_IntroHTS}) and (\ref{Eq_IntroRTS}), $F$ and $\widetilde{F}$ are
sub-probability distribution functions on $[0,\infty)$, and we write
$F_{l}(t)\Rightarrow F(t)$ to indicate that $F_{l}(t)\rightarrow F(t)$ at
every continuity point $t$ of $F$. As reviewed below (Theorem
\ref{T_HittingTimeVsReturnTimeLimits}), assertions (\ref{Eq_IntroHTS}) and
(\ref{Eq_IntroRTS}) are intimately related to each other. There is a large
body of literature discussing these (and related) types of limit theorems,
both in the setup of abstract ergodic theory, and in the framework of specific
families of dynamical systems. A particularly prominent situation is that of
convergence to a \emph{normalized exponential law}, meaning that
$F(t)=\widetilde{F}(t)=1-e^{-t}$ for $t\geq0$ (which is the only case where
$F=\widetilde{F}$).%

%TCIMACRO{\TeXButton{VSs}{\vspace{0.3cm}}}%
%BeginExpansion
\vspace{0.3cm}%
%EndExpansion

The present paper discusses the question whether distributional limit theorems
of the type (\ref{Eq_IntroHTS}) or (\ref{Eq_IntroRTS}) can be improved to
obtain sharper \emph{local limit theorems} (\emph{LLTs}) for the
integer-valued variables $\varphi_{A_{l}}$, identifying the asymptotics, as
$l\rightarrow\infty$, of $\mu(\varphi_{A_{l}}=k_{l})$ and $\mu_{A_{l}}%
(\varphi_{A_{l}}=k_{l})$, respectively, for individual integer values $k_{l}$
on the correct scale. It seems that no results of this type are available so far.%

%TCIMACRO{\TeXButton{VSs}{\vspace{0.3cm}}}%
%BeginExpansion
\vspace{0.3cm}%
%EndExpansion

Slightly extending the standard notation for asymptotic equivalence, we write
$a_{n}\sim c\cdot b_{n}$ to indicate that $b_{n}>0$ for large $n$, and
$a_{n}/b_{n}\rightarrow c\in\lbrack0,\infty)$. (This defines a relation $\sim
c\,\cdot$. In case $c=0$, $a_{n}\sim c\cdot b_{n}$ means that $a_{n}=o(b_{n}%
)$). We shall provisionally say that $(A_{l})$ satisfies a \emph{local limit
theorem for hitting times} if
\begin{equation}
\mu(\varphi_{A_{l}}=k_{l})\sim F^{\prime}(t)\cdot\mu(A_{l})\text{ \quad as
}l\rightarrow\infty
\end{equation}
whenever the $k_{l}$ are integers which satisfy $\mu(A_{l})k_{l}\rightarrow
t\ $for (suitable) $t>0$, and that $(A_{l})$ satisfies a \emph{local limit
theorem for return times} if, for such $(k_{l})$,%
\begin{equation}
\mu_{A_{l}}(\varphi_{A_{l}}=k_{l})\sim\widetilde{F}^{\prime}(t)\cdot\mu
(A_{l})\text{ \quad as }l\rightarrow\infty\text{.}%
\end{equation}
(To see that the derivatives on the right-hand sides make sense, recall that
any limit $F$ of hitting time distributions as in (\ref{Eq_IntroHTS}) is
necessarily concave and (hence) absolutely continuous with a Lebesgue almost
everywhere defined derivative (see Theorem
\ref{T_HittingTimeVsReturnTimeLimits} and Remark \ref{Rem_PropF}\ below for
details). For definiteness, we let $F^{\prime}$ denote the unique
right-continuous version of the derivative which has left-hand limits
everywhere. This is a non-increasing function on $[0,\infty)$. Regarding
return-times, we shall focus on situations in which $\widetilde{F}^{\prime}$
can be given explicitly.) It will be convenient to have a shorthand notation
for the collection of sequences $(k_{l})$ of integers which asymptotically
have the correct order of magnitude in that $(\mu(A_{l})k_{l})_{l\geq1}$ is
bounded away from $0$ and $\infty$. We therefore write, for $\delta\in(0,1]$,
\begin{gather}
\mathcal{K}_{\delta}(A_{l}):=\left\{  (k_{l})_{l\geq1}:k_{l}\in\mathbb{N}%
,\,\delta\leq\mu(A_{l})k_{l}\leq\delta^{-1}\text{ for }l\geq1\right\}
\text{,}\nonumber\\
\text{and \quad}\mathcal{K}(A_{l}):=%
%TCIMACRO{\tbigcup \nolimits_{\delta\in(0,1]}}%
%BeginExpansion
{\textstyle\bigcup\nolimits_{\delta\in(0,1]}}
%EndExpansion
\mathcal{K}_{\delta}(A_{l}).\label{Eq_DefCurlyK}%
\end{gather}
%

%TCIMACRO{\TeXButton{VSs}{\vspace{0.3cm}}}%
%BeginExpansion
\vspace{0.3cm}%
%EndExpansion
%

%TCIMACRO{\TeXButton{NI}{\noindent}}%
%BeginExpansion
\noindent
%EndExpansion
\textbf{Main results.} It turns out that distributional convergence of
\emph{hitting times} automatically entails a corresponding LLT in many
situations (Theorem \ref{T_LLT_Hit} below). In the special case of an
exponential limit, this result becomes

\begin{theorem}
[\textbf{Automatic LLT for exponential hitting times}]\label{T_AutoLLT_Hit}Let
$T$ be an ergodic measure-preserving map on the probability space
$(X,\mathcal{A},\mu)$, and $(A_{l})_{l\geq1}$ a sequence in $\mathcal{A}$ such
that $0<\mu(A_{l})\rightarrow0$. Then the normalized hitting times converge to
an exponential law with expectation $1/\theta$ (necessarily with $\theta
\in(0,1]$),
\begin{equation}
\mu(\mu(A_{l})\varphi_{A_{l}}\leq t)\Longrightarrow1-e^{-\theta t}\text{ \quad
as }l\rightarrow\infty\text{,}%
\end{equation}
if and only if for every $\delta\in(0,1]$,
\begin{equation}
\mu(\varphi_{A_{l}}=k_{l})\sim\theta e^{-\theta\mu(A_{l})k_{l}}\,\mu
(A_{l})\text{ \quad}%
\begin{array}
[c]{c}%
\text{as }l\rightarrow\infty\text{,}\\
\text{uniformly in }(k_{l})\in\mathcal{K}_{\delta}(A_{l})\text{.}%
\end{array}
\label{Eq_AutoLLT_HitExp}%
\end{equation}
In particular,%
\begin{gather}
\mu(\mu(A_{l})\varphi_{A_{l}}\leq t)\Longrightarrow1-e^{-t}\text{ as
}l\rightarrow\infty\text{\ \quad holds if and only if}\\
\mu(\varphi_{A_{l}}=k_{l})\sim e^{-\mu(A_{l})k_{l}}\,\mu(A_{l})\text{\quad
uniformly in }(k_{l})\in\mathcal{K}_{\delta}(A_{l})\text{ for all }%
\delta>0\text{.}\nonumber
\end{gather}

\end{theorem}

\begin{remark}
[\textbf{A comment on the parameter} $\theta$]\label{Rem_CommentTheta}The most
common case is that of a normalized exponential law ($\theta=1$). In many
systems with some hyperbolicity, this is what happens if $(A_{l})$ is a
sequence of sufficiently regular sets which shrink to a typical point
$x^{\ast}\in X$. In the same setup, different pairs $(F,\widetilde{F})$ of
limit laws can arise in (\ref{Eq_IntroHTS}) and (\ref{Eq_IntroRTS}) when
$x^{\ast}$ is, say, a fixed point, where one often (but not always, see
Example 4.2 of \cite{Zexceptional}) obtains $F(t)=1-e^{-\theta t}$ and
$\widetilde{F}(t)=1-\theta e^{-\theta t}$ for $t\geq0$, with $\theta\in(0,1)$.
In more abstract terms, we can express this standard scenario by asserting
that $A_{l}=A_{l}^{\bullet}\cup A_{l}^{\circ}$ (measurable and disjoint) where
the points of $A_{l}^{\bullet}=A_{l}\cap T^{-1}A_{l}$ instantly return to
$A_{l}$. In particular, $(\varphi_{A_{l}})_{l\geq1}\ $is bounded on
$A_{l}^{\bullet}$, and $\theta$ is obtained as the limit of $\mu_{A_{l}}%
(A_{l}^{\circ})\ as$ $l\rightarrow\infty$. (Also see the examples below, and
Theorem 3.8 of \cite{ZWhenAndWhere} for a functional version.)
\end{remark}%

%TCIMACRO{\TeXButton{VSs}{\vspace{0.3cm}}}%
%BeginExpansion
\vspace{0.3cm}%
%EndExpansion

The case of \emph{return times} is more subtle, and we will focus on the case
of exponential limit laws there. But even then, $\mu_{A_{l}}(\mu(A_{l}%
)\varphi_{A_{l}}\leq t)\Longrightarrow1-e^{-t}$ does not in general ensure
that $\mu_{A_{l}}(\varphi_{A_{l}}=k_{l})\sim e^{-\mu(A_{l})k_{l}}\,\mu(A_{l})$
for all $(k_{l})\in\mathcal{K}(A_{l})$. We will, however, provide sufficient
abstract conditions for validity of the latter (Theorem \ref{T_LLT_Ret0}).
With the aid of an extended version of that result (Theorem
\ref{T_LLT_RetGeneral}), we will in fact establish \emph{multidimensional
local limit theorems for both return times and hitting times}, which also
allow us to \emph{include a spatial component}. See Theorem
\ref{T_LLT_ConsecutiveTimes} which is the main abstract result of the present
paper. In Section \ref{Sec_ConcreteLLTs} these general results are then
applied in some concrete standard situations (Gibbs-Markov systems). Those
include the following classical

\begin{example}
[\textbf{Large digits in continued fraction expansions}]\label{Ex_CF}We
illustrate our main results in the prominent special case of continued
fraction expansions. Set $X:=[0,1]$, $\mathcal{A}:=\mathcal{B}_{X}$ (the Borel
$\sigma$-algebra), and let $T:X\rightarrow X$ be \emph{Gauss' continued
fraction map} with $T0:=0$ and
\begin{equation}
Tx:=\dfrac{1}{x}-\left\lfloor \dfrac{1}{x}\right\rfloor =\dfrac{1}{x}-k\text{
\qquad for }x\in\left(  \dfrac{1}{k+1},\dfrac{1}{k}\right]  =:I_{k}\text{,
}k\geq1\text{,}%
\end{equation}
which preserves the probability density $h(x):=\frac{1}{\log2}\frac{1}{1+x}$,
$x\in X$. The invariant measure $\mu$ thus defined, $\mu(A):=\int
_{A}h\,d\lambda$ with $\lambda$ denoting Lebesgue measure on $X$, is exact
(and hence ergodic). Iteration of $T$ reveals the \emph{continued fraction
(CF) digits} of any $x\in X$, in that
\begin{equation}
x=\frac{1}{\mathsf{a}_{1}(x)+\dfrac{1}{\mathsf{a}_{2}(x)+\cdots}}\text{ \quad
with \quad}\mathsf{a}_{n}(x)=\mathsf{a}\circ T^{n-1}(x)\text{, }n\geq1\text{,}%
\end{equation}
where $\mathsf{a}:X\rightarrow\mathbb{N}$ is the \emph{digit function}
corresponding to $\xi:=\{I_{k}:k\geq1\}$, i.e. $\mathsf{a}(x):=\left\lfloor
1/x\right\rfloor =k$ for $x\in I_{k}$. It is a standard fact that the ergodic
measure preserving \emph{CF-system} $(X,\mathcal{A},\mu,T,\xi)$ is
Gibbs-Markov (see Section \ref{Sec_ConcreteLLTs}).

Let $\tau_{l}^{(1)}(x):=\inf\{n\geq1:\mathsf{a}_{n}(x)\geq l\}$ denote the
position of the first CF-digit no less than $l$, and define $\tau_{l}%
^{(j+1)}(x):=\inf\{n\geq1:\mathsf{a}_{\tau_{l}^{(1)}(x)+\cdots+\tau_{l}%
^{(j)}(x)+n}(x)\geq l\}$, the distance between the positions of the $j$th and
the $(j+1)$st digit in that range. A classical result with non-trivial history
(\cite{Doeblin}, \cite{I}) asserts that the positions of large digits converge
in distribution, as $l\rightarrow\infty$, to a Poisson process. Equivalently,
the sequence of consecutive distances converges to an iid sequence of
exponential random variables in that for every $d\geq1$ and $t^{(1)}%
,\ldots,t^{(d)}\geq0$,
\begin{equation}
\mu\left(  \bigcap_{j=1}^{d}\left\{  \frac{\tau_{l}^{(j)}}{l\log2}%
>t^{(j)}\right\}  \right)  \longrightarrow\prod_{j=1}^{d}e^{-t^{(j)}}\text{
\quad as }l\rightarrow\infty\text{.}\label{Eq_CF_LLT0}%
\end{equation}
Consider the values $\psi_{l}^{(j)}(x):=\mathsf{a}_{\tau_{l}^{(1)}%
(x)+\cdots+\tau_{l}^{(j)}(x)}(x)$ of the large digits thus observed. A refined
version of (\ref{Eq_CF_LLT0}) which asserts that the $\psi_{l}^{(j)}/l$ also
converge in law and are asymptotically independent of the times $\tau
_{l}^{(j)}$ has been given in Proposition 10.4 of \cite{ZWhenAndWhere}.
Theorem \ref{T_LLTGMCyls2} below implies a local version of the latter
statement: For every $d\geq1$, any sequences $(k_{l}^{(1)}),\ldots
,(k_{l}^{(d)})$ of positive integers with $(k_{l}^{(j)}/l)_{l\geq1}$ bounded
away from $0$ and $\infty$, and any sequences $(a_{l}^{(1)}),\ldots
,(a_{l}^{(d)})$ of integers $a_{l}^{(j)}\geq l$ we obtain%
\begin{equation}
\mu\left(  \bigcap_{j=1}^{d}\left\{  \tau_{l}^{(j)}=k_{l}^{(j)},\psi_{l}%
^{(j)}=a_{l}^{(j)}\right\}  \right)  \sim\prod_{j=1}^{d}\frac{e^{-\frac
{k_{l}^{(j)}}{l\log2}}}{(a_{l}^{(j)})^{2}\,\log2}\text{ \quad as }%
l\rightarrow\infty\text{.}\label{Eq_CF_LLT1}%
\end{equation}

Amusing variations follow via the same result. For instance, let
$\overline{\tau}_{l}^{(1)}(x):=\inf\{n\geq1:\mathsf{a}_{n}(x)\geq l$ and
$\mathsf{a}_{n}(x)$ prime$\}$ denote the position of the first CF-digit no
less than $l$ which is a prime number, and define $\overline{\tau}_{l}%
^{(j+1)}(x):=\inf\{n\geq1:\mathsf{a}_{\overline{\tau}_{l}^{(1)}(x)+\cdots
+\overline{\tau}_{l}^{(j)}(x)+n}(x)\geq l$ and prime$\}$. Write $\overline
{\psi}_{l}^{(j)}(x):=\mathsf{a}_{\overline{\tau}_{l}^{(1)}(x)+\cdots
+\overline{\tau}_{l}^{(j)}(x)}(x)$. Then, for every $d\geq1$, any sequences
$(k_{l}^{(1)}),\ldots,(k_{l}^{(d)})$ of positive integers with $(k_{l}%
^{(j)}/l\,\log l)_{l\geq1}$ bounded away from $0$ and $\infty$, and any
sequences $(a_{l}^{(1)}),\ldots,(a_{l}^{(d)})$ of prime numbers $a_{l}%
^{(j)}\geq l$ we get%
\begin{equation}
\mu\left(  \bigcap_{j=1}^{d}\left\{  \overline{\tau}_{l}^{(j)}=k_{l}%
^{(j)},\overline{\psi}_{l}^{(j)}=a_{l}^{(j)}\right\}  \right)  \sim\prod
_{j=1}^{d}\frac{e^{-\frac{k_{l}^{(j)}}{l\,\log l\,\log2}}}{(a_{l}^{(j)}%
)^{2}\,\log2}\text{ \quad as }l\rightarrow\infty\text{.}\label{Eq_CF_LLTprime}%
\end{equation}
Moreover, the Gauss measure $\mu$ can be replaced by Lebesgue measure
$\lambda$ in (\ref{Eq_CF_LLT1}) and (\ref{Eq_CF_LLTprime}).
\end{example}%

%TCIMACRO{\TeXButton{VSs}{\vspace{0.3cm}}}%
%BeginExpansion
\vspace{0.3cm}%
%EndExpansion
%

%TCIMACRO{\TeXButton{NI}{\noindent}}%
%BeginExpansion
\noindent
%EndExpansion
\textbf{Acknowledgement.} RZ thanks Beno\^{\i}t Saussol for crucial inspiring
and enlightening conversations on this topic. We are also grateful to the
careful referee. This research was supported by the Austrian Science Fund
(FWF): P 33943-N.%

%TCIMACRO{\TeXButton{VSs}{\vspace{0.3cm}}}%
%BeginExpansion
\vspace{0.3cm}%
%EndExpansion

\section{Abstract Local Limit Theorems\label{Sec_AbstractLLTs}}

To begin our investigation, we first study LLTs in the general abstract setup
of ergodic probability preserving maps.%

%TCIMACRO{\TeXButton{VSs}{\vspace{0.3cm}}}%
%BeginExpansion
\vspace{0.3cm}%
%EndExpansion
%

%TCIMACRO{\TeXButton{NI}{\noindent}}%
%BeginExpansion
\noindent
%EndExpansion
\textbf{A general LLT for hitting times.} To formulate the general result, let
$\lambda$ denote one-dimensional Lebesgue measure, and $\mathcal{B}_{E}$ the
Borel-$\sigma$-algebra of the space $E$. Measurable functions are
\emph{version}s of each other if they coincide almost everywhere.

\begin{theorem}
[\textbf{Automatic local limit theorem for hitting times}]\label{T_LLT_Hit}Let
$T$ be an ergodic measure-preserving map on the probability space
$(X,\mathcal{A},\mu)$, and $(A_{l})_{l\geq1}$ a sequence in $\mathcal{A}$ such
that $0<\mu(A_{l})\rightarrow0$.\newline\newline\textbf{(a)} Suppose that the
normalized hitting times converge to some sub-probability distribution
function $F$ on $[0,\infty)$,
\begin{equation}
\mu(\mu(A_{l})\varphi_{A_{l}}\leq t)\Longrightarrow F(t)\text{ \quad as
}l\rightarrow\infty\text{.}\label{Eq_DLT_HTS}%
\end{equation}
Then, for sequences $(k_{l})_{l\geq1}$ of integers,
\begin{gather}
\mu(\varphi_{A_{l}}=k_{l})\sim F^{\prime}(t)\cdot\mu(A_{l})\text{ \quad as
}l\rightarrow\infty\label{Eq_LLT_HTS_Simple}\\
\text{whenever }\mu(A_{l})k_{l}\rightarrow t>0\text{ and }F^{\prime}\text{ is
continuous at }t\text{.}\nonumber
\end{gather}
\textbf{(b)} Conversely, assume that $D\in\mathcal{B}_{(0,\infty)}$ satisfies
$\lambda((0,\infty)\setminus D)=0$ and $G:D\rightarrow\lbrack0,\infty)$ is
such that, for sequences $(k_{l})_{l\geq1}$ of integers,
\begin{equation}
\mu(\varphi_{A_{l}}=k_{l})\sim G(t)\cdot\mu(A_{l})\text{\quad as }%
l\rightarrow\infty\quad\text{whenever }\mu(A_{l})k_{l}\rightarrow t\in
D\text{.}\label{Eq_AssmLLTHitConverse}%
\end{equation}
Then (\ref{Eq_DLT_HTS}) holds with $F(t):=\int_{0}^{t}G(s)\,ds$,
$t>0$.\newline\newline\textbf{(c)} Under (\ref{Eq_DLT_HTS}), if $F^{\prime}$
is continuous and strictly positive on an open interval $I\subseteq(0,\infty
)$, then, for every compact subset $J$ of $I$,
\begin{equation}
\mu(\varphi_{A_{l}}=k_{l})\sim F^{\prime}(\mu(A_{l})k_{l})\cdot\mu
(A_{l})\text{ \quad}%
\begin{array}
[c]{c}%
\text{as }l\rightarrow\infty\text{,}\\
\text{uniformly in }(k_{l})\in\mathcal{K}_{J}(A_{l})\text{,}%
\end{array}
\label{Eq_LLT_HTS_cpt1}%
\end{equation}
where $\mathcal{K}_{J}(A_{l}):=\left\{  (k_{l})_{l\geq1}:k_{l}\in
\mathbb{N},\,\mu(A_{l})k_{l}\in J\text{ for }l\geq1\right\}  $.
\end{theorem}

\begin{remark}
\textbf{a)} Note the absence of explicit conditions akin to aperiodicity (weak
mixing, say) of the map $T$. \newline\textbf{b)} This result contains, as an
easy special case, Theorem \ref{T_AutoLLT_Hit} above.\newline\textbf{c)}
Recall that by our convention regarding $a_{n}\sim c\cdot b_{n}$ the above
also covers situations with $F^{\prime}(t)=0$ or $G(t)=0$, in which cases
(\ref{Eq_LLT_HTS_Simple}) and (\ref{Eq_AssmLLTHitConverse}) say that
$\mu(\varphi_{A_{l}}=k_{l})=o(\mu(A_{l}))$.
\end{remark}

We shall also see, as a consequence, that LLTs for hitting times are
remarkably robust. To formulate this concisely, we first record

\begin{remark}
In the situation of Theorem \ref{T_LLT_Hit} (b) we can, in a second step,
apply (a) to see that $G$ has a non-increasing c\`{a}dl\`{a}g version (namely
$F^{\prime}$) for which (\ref{Eq_AssmLLTHitConverse}) holds, specifically,
with $D$ the set of its continuity points in $(0,\infty)$.
\end{remark}

Hence, assuming these properties of $G$ and $D$ does not restrict
applicability of

\begin{theorem}
[\textbf{Nearby sequences inherit LLT for hitting times}]%
\label{T_RobustLLThit}Let $T$ be an ergodic measure-preserving map on the
probability space $(X,\mathcal{A},\mu)$, and let $(A_{l}),(B_{l})$ be
sequences in $\mathcal{A}$ such that $0<\mu(A_{l})\rightarrow0$ and $\mu
(A_{l}\triangle B_{l})=o(\mu(A_{l}))$. Suppose that $G$ is a non-increasing
c\`{a}dl\`{a}g function on $[0,\infty)$ with continuity set $D$ such that, for
sequences $(k_{l})_{l\geq1}$ of integers,%
\begin{equation}
\mu(\varphi_{A_{l}}=k_{l})\sim G(t)\cdot\mu(A_{l})\text{\quad as }%
l\rightarrow\infty\quad\text{whenever }\mu(A_{l})k_{l}\rightarrow t\in
D\text{.}\label{T_bchscbhdsbchsbch}%
\end{equation}
Then (\ref{T_bchscbhdsbchsbch}) remains true with $A_{l}$ replaced by $B_{l}$.
\end{theorem}%

%TCIMACRO{\TeXButton{VSs}{\vspace{0.3cm}}}%
%BeginExpansion
\vspace{0.3cm}%
%EndExpansion
%

%TCIMACRO{\TeXButton{NI}{\noindent}}%
%BeginExpansion
\noindent
%EndExpansion
\textbf{LLT for hitting times under different measures.} It is an interesting
and useful general fact that distributional convergence $\mu(\mu(A_{l}%
)\varphi_{A_{l}}\leq t)\Longrightarrow F(t)$ of normalized hitting times in an
ergodic m.p. system always is an instance of \emph{strong distributional
convergence}, meaning that it remains valid whenever the invariant measure is
replaced by some absolutely continuous probability. The following is contained
in Corollary 5 of \cite{Z7}. \

\begin{proposition}
[\textbf{Strong distributional convergence of hitting-times}]%
\label{P_StrongDistrCgeHittingTimes}\negthinspace Let $(X,\mathcal{A},\mu,T) $
be ergodic and probability-preserving, $(A_{l})_{l\geq1}$ a sequence of
asymptotically rare events, and $F$ a sub-probability distribution function on
$[0,\infty)$. If
\begin{equation}
\nu(\mu(A_{l})\varphi_{A_{l}}\leq t)\Longrightarrow F(t)\text{ \quad as
}l\rightarrow\infty\label{Eq_vdhbvjdbvjhbvjbyvjhbdvjbbbbbbbbbb}%
\end{equation}
for \emph{some} probability measure $\nu\ll\mu$, then
(\ref{Eq_vdhbvjdbvjhbvjbyvjhbdvjbbbbbbbbbb}) holds for \emph{all}
probabilities $\nu\ll\mu$.
\end{proposition}

While Theorem \ref{T_LLT_Hit} shows that distributional convergence under
$\mu$ automatically upgrades to local convergence under $\mu$ at continuity
points of $F^{\prime}$, it is easy to see that the statement analogous to
Theorem \ref{T_LLT_Hit} fails for general $\nu\ll\mu$:

For example, if the system has a cyclic structure in that $C,T^{-1}%
C,\ldots,T^{-p+1}C$ are pairwise disjoint with $T^{-p}C=C$ for some
$C\in\mathcal{A} $ with $\mu(C)>0$ and $p\geq2$, and $A_{l}\subseteq C$ for
all $l$, then the values of the $\varphi_{A_{l}}$ are always multiples of $p$
on $C$, and taking $\nu:=\mu_{C}$ we have $\nu(\varphi_{A_{l}}=k_{l})=0$
unless $p$ divides $k_{l}$.

But periodicity is not the only possible obstacle. In fact, in every system,
and for every sequence $(A_{l})$ there is some $\nu$ which fails the LLT in a
similar manner.

\begin{proposition}
[\textbf{Initial distributions foiling LLT for hitting times}]%
\label{P_BadInitialDistrForLLTHit}Let $T$ be ergodic and measure preserving on
the probability space $(X,\mathcal{A},\mu)$, and $A_{l}\in\mathcal{A}$ such
that $0<\mu(A_{l})\rightarrow0$ and
\begin{equation}
\mu(\mu(A_{l})\varphi_{A_{l}}\leq t)\Longrightarrow F(t)\text{ \quad as
}l\rightarrow\infty\label{Eq_yxcvvcxy}%
\end{equation}
for some sub-probability distribution function $F$. Assume that $F^{\prime}$
is continuous at $s>0$ with $F^{\prime}(s)>0$, and take integers $k_{l}$ with
$\mu(A_{l})k_{l}\rightarrow s$. Then there is some $B\in\mathcal{A}$ (where
$\mu(B^{c})>0$ can be chosen arbitrarily small) such that $\nu:=\mu_{B}\ll\mu$
satisfies
\begin{equation}
\nu(\mu(A_{l})\varphi_{A_{l}}\leq t)\Longrightarrow F(t)\text{ \quad as
}l\rightarrow\infty\text{,}\label{Eq_dfgvdfvjhacvvvvvvvvvvvvvvvvv}%
\end{equation}
while \
\begin{equation}
\nu(\varphi_{A_{l}}=k_{l})=0\text{ \quad for infinitely many }l\geq
1\text{.}\label{Eq_yxcvvcxyLLTfails}%
\end{equation}

\end{proposition}

Nonetheless, we can provide easy sufficient conditions under which the LLT for
hitting times also holds under a specific probability measure $\nu$ different
from the invariant measure $\mu$. It is clear that we must at least rule out
periodic behaviour.

If $\nu\ll\mu$, suitable conditions can be expressed in terms of the density
of $\nu$. We let $\mathcal{D}(\mu)$ denote the set of probability densities
$u$ w.r.t. $\mu$, and set $u\odot\mu(A):=\int_{A}u\,d\mu$. The \emph{transfer
operator} $\widehat{T}:L_{1}(\mu)\rightarrow L_{1}(\mu)$ of $T$ on
$(X,\mathcal{A},\mu)$ describes the evolution of probability densities under
$T$. That is, if $\nu$ has density $u$ w.r.t. $\mu$, so that $u=d\nu/d\mu$,
then $\widehat{T}u:=d(\nu\circ T^{-1})/d\mu$. Equivalently, $\int(g\circ
T)\cdot u\,d\mu=\int g\cdot\widehat{T}u\,d\mu$ for all $u\in L_{1}(\mu)$ and
$g\in L_{\infty}(\mu)$.

When combined with Theorem \ref{T_LLT_Hit}, the following result extends the
latter to a larger family of measures $\nu$. Assuming $\mu(\varphi_{A_{l}%
}=k_{l})\sim F^{\prime}(\mu(A_{l})k_{l})\cdot\mu(A_{l})$ we obviously also get
$\nu(\varphi_{A_{l}}=k_{l})\sim F^{\prime}(\mu(A_{l})k_{l})\cdot\mu(A_{l})$ as
soon as $\nu$ satisfies%
\begin{equation}
\left\vert \nu(\varphi_{A_{l}}=k_{l})-\mu(\varphi_{A_{l}}=k_{l})\right\vert
=o(\mu(A_{l}))\text{ \quad as }l\rightarrow\infty\text{.}%
\label{Eq_LLTgoodBaby}%
\end{equation}
The next proposition actually asserts uniformity of convergence in the LLT on
certain classes of measures, which will be an important step in the approach
to LLTs for \emph{return} times discussed below. A family $\mathcal{V}$ of
probability measures $\nu$ on $(X,\mathcal{A})$ will be called
\emph{LLT-uniform for} $(A_{l})$ if for every $(k_{l})\in\mathcal{K}(A_{l})$
(recall (\ref{Eq_DefCurlyK})),
\begin{equation}
\left\vert \nu(\varphi_{A_{l}}=k_{l})-\mu(\varphi_{A_{l}}=k_{l})\right\vert
=o(\mu(A_{l}))\text{ \quad as }l\rightarrow\infty\text{, uniformly in }\nu
\in\mathcal{V}\label{Eq_LLTgood}%
\end{equation}
(meaning $\left\vert \nu(\varphi_{A_{l}}=k_{l})-\mu(\varphi_{A_{l}}%
=k_{l})\right\vert /\mu(A_{l})\rightarrow0$ uniformly in $\nu\in\mathcal{V}$).
Note that we can assume w.l.o.g. that $\mu\in\mathcal{V}$. We say that a
collection $\mathcal{U}\subseteq\mathcal{D}(\mu)$ of probability densities is
\emph{LLT-uniform for} $(A_{l})$ if the family $\{u\odot\mu:u\in\mathcal{U}\}$
of measures is. We can then state

\begin{proposition}
[\textbf{Change of measure in LLT for hitting times}]%
\label{T_LLT_Hit_DensiFromU}Let $T$ be an ergodic measure-preserving map on
the probability space $(X,\mathcal{A},\mu)$, and $(A_{l})_{l\geq1}$ a sequence
in $\mathcal{A}$ such that $0<\mu(A_{l})\rightarrow0$. Suppose that the
normalized hitting times converge to some sub-probability distribution
function $F$ on $[0,\infty)$ with $F^{\prime}$ continuous and strictly
positive on $(0,\infty)$,
\begin{equation}
\mu(\mu(A_{l})\varphi_{A_{l}}\leq t)\Longrightarrow F(t)\text{ \quad as
}l\rightarrow\infty\text{.}\label{Eq_gvdcgsvggggggy}%
\end{equation}
\textbf{(a)} Let $\mathcal{V}$ be a family of probability measures on
$(X,\mathcal{A})$. If $\mathcal{V}$ is LLT-uniform for $(A_{l})$, then for
every $\delta\in(0,1]$,%
\begin{equation}
\nu(\varphi_{A_{l}}=k_{l})\sim F^{\prime}(\mu(A_{l})k_{l})\cdot\mu
(A_{l})\text{ \quad}%
\begin{array}
[c]{c}%
\text{as }l\rightarrow\infty\text{, uniformly in}\\
\nu\in\mathcal{V}\text{ and }(k_{l})\in\mathcal{K}_{\delta}(A_{l})\text{.}%
\end{array}
\label{Eq_LLT_HTS_cpt_UVV}%
\end{equation}
\textbf{(b)} Let $\mathcal{U}\subseteq\mathcal{D}(\mu)$ be such that
$\widehat{T}\mathcal{U}$ is bounded in $L_{\infty}(\mu)$ and
\begin{equation}
\parallel\widehat{T}^{n}u-1_{X}\parallel_{\infty}\longrightarrow0\text{ \quad
as }n\rightarrow\infty\text{, uniformly in }u\in\mathcal{U}\text{.}%
\label{Eq_UnifUnifCge}%
\end{equation}
Then $\mathcal{U}$ is LLT-uniform for $(A_{l})$.
\end{proposition}

\begin{remark}
\textbf{a)} Observe that the definition of an \emph{LLT-uniform} family
$\mathcal{V}$ does not a priori require the convergence in (\ref{Eq_LLTgood})
to be uniform on $\mathcal{K}_{\delta}(A_{l})$. But statement (a) of the
proposition shows that, under (\ref{Eq_gvdcgsvggggggy}), it is automatic that
(\ref{Eq_LLTgood}) also holds uniformly in $(k_{l})\in\mathcal{K}_{\delta
}(A_{l})$. \newline\textbf{b)} An analogous statement holds for general $F$
and sequences $(k_{l})$ of integers for which either $\mu(A_{l})k_{l}%
\rightarrow t$ (some continuity point $t>0$ of $F^{\prime}$ with $F^{\prime
}(t)>0$), or such that $(\mu(A_{l})k_{l})_{l\geq1}$ is contained in some
compact subset of $I $, where $F$ is strictly positive and $\mathcal{C}^{1}$
on the open interval $I\subseteq(0,\infty)$. (The argument remains the
same.)\newline\textbf{c)} The possibility of changing the initial distribution
granted by this proposition immediately translates into a simple conditional
LLT (for a more sophisticated version see Theorem \ref{T_LLT_ConsecutiveTimes}
below): Under the assumptions of Proposition \ref{T_LLT_Hit_DensiFromU} a)
suppose that $(B_{l})\subseteq\mathcal{A}$ is a sequence of conditioning
events (with $\mu(B_{l})>0$) such that $\mu_{B_{l}}\in\mathcal{V}$ for all
$l$. Then, for every $(k_{l})\in\mathcal{K}(A_{l})$, we have $\mu_{B_{l}%
}(\varphi_{A_{l}}=k_{l})\sim F^{\prime}(\mu(A_{l})k_{l})\cdot\mu(A_{l})$ as
$l\rightarrow\infty$.
\end{remark}%

%TCIMACRO{\TeXButton{VSs}{\vspace{0.3cm}}}%
%BeginExpansion
\vspace{0.3cm}%
%EndExpansion
%

%TCIMACRO{\TeXButton{NI}{\noindent}}%
%BeginExpansion
\noindent
%EndExpansion
\textbf{A LLT for asymptotically exponential return times.} Having discussed
the passage from distributional convergence of \emph{hitting times}, say
$\mu(\mu(A_{l})\varphi_{A_{l}}\leq t)\Longrightarrow1-e^{-t}$, to its local
version $\mu(\varphi_{A_{l}}=k_{l})\sim e^{-\mu(A_{l})k_{l}}\,\mu(A_{l})$, we
can just as well ask whether a corresponding distributional limit theorem for
\emph{return times}, such as $\mu_{A_{l}}(\mu(A_{l})\varphi_{A_{l}}\leq
t)\Longrightarrow1-e^{-t}$, implies a LLT like $\mu_{A_{l}}(\varphi_{A_{l}%
}=k_{l})\sim e^{-\mu(A_{l})k_{l}}\,\mu(A_{l})$.

This is not even true in the case of an exponential limit law, as is easily
seen in cyclic situations like those mentioned before Proposition
\ref{P_BadInitialDistrForLLTHit}. Worse yet, it takes very little to sabotage
an LLT in an arbitrary system.The following observation is in stark contrast
to Theorem \ref{T_RobustLLThit}.

\begin{proposition}
[\textbf{Nearby sequences which fail LLT for return times}]\label{P_ookmmmmms}%
Let $T$ be an ergodic measure-preserving map on the probability space
$(X,\mathcal{A},\mu)$, and $(A_{l})_{l\geq1}$ a sequence in $\mathcal{A}$ such
that $0<\mu(A_{l})\rightarrow0$. Suppose that
\begin{equation}
\mu_{A_{l}}(\mu(A_{l})\varphi_{A_{l}}\leq t)\Longrightarrow\widetilde
{F}(t)\text{ \quad as }l\rightarrow\infty\text{,}\label{Eq_kjhgfdsalkjhgfds}%
\end{equation}
for some distribution function $\widetilde{F}$. Take integers $k_{l}$ for
which $\mu(A_{l})k_{l}\ $converges to a continuity point of $\widetilde{F}$.
Then there exist $B_{l}\subseteq A_{l}$ such that $\mu(B_{l})\sim\mu(A_{l})$
and
\begin{equation}
\mu_{B_{l}}(\mu(B_{l})\varphi_{B_{l}}\leq t)\Longrightarrow\widetilde
{F}(t)\text{ \quad as }l\rightarrow\infty\text{,}\label{Eq_zhbgzhbgzhbgzhbg}%
\end{equation}
while
\begin{equation}
\mu_{B_{l}}(\varphi_{B_{l}}=k_{l})=0\text{ \quad for }l\geq1\text{.}%
\end{equation}

\end{proposition}%

%TCIMACRO{\TeXButton{VSs}{\vspace{0.3cm}}}%
%BeginExpansion
\vspace{0.3cm}%
%EndExpansion

We can, however, identify conditions which do imply an LLT for return times.
The assumptions of the following theorem ensure that the local behaviour of
return times is essentially equivalent to that of hitting times. (For this
reason, our approach to the LLT for return times is limited to situations with
exponential limit laws.) To cover exponential limits with general expectation
$1/\theta$, we formulate the assumptions in the typical $A_{l}=A_{l}^{\bullet
}\cup A_{l}^{\circ}$ situation of Remark \ref{Rem_CommentTheta} in which
$\theta\neq1$ may arise naturally.

\begin{theorem}
[\textbf{LLT for exponential return times}]\label{T_LLT_Ret0}Let $T$ be an
ergodic measure-preserving map on the probability space $(X,\mathcal{A},\mu)$,
$\theta\in(0,1]$, and $(A_{l})_{l\geq1}$ a sequence in $\mathcal{A}$ such that
$0<\mu(A_{l})\rightarrow0$ and $A_{l}=A_{l}^{\bullet}\cup A_{l}^{\circ}$
(measurable and disjoint) with
\begin{equation}
\mu_{A_{l}}(A_{l}^{\circ})\rightarrow\theta\text{ \quad and \quad}\mu
(A_{l})\parallel1_{A_{l}^{\bullet}}\varphi_{A_{l}}\parallel_{\infty
}\rightarrow0\text{ \quad as }l\rightarrow\infty\text{.}%
\label{Eq_ReturningPart}%
\end{equation}
Suppose that the normalized hitting times converge to an exponential law,
\begin{equation}
\mu(\mu(A_{l})\varphi_{A_{l}}\leq t)\Longrightarrow1-e^{-\theta t}\text{ \quad
as }l\rightarrow\infty\text{.}\label{Eq_ExpHitAssm}%
\end{equation}
\textbf{(a)} If for every sequence $(k_{l})\in\mathcal{K}(A_{l})$,%
\begin{equation}
\left\vert \mu_{A_{l}^{\circ}}(\varphi_{A_{l}}=k_{l})-\mu(\varphi_{A_{l}%
}=k_{l})\right\vert =o(\mu(A_{l}))\text{ \quad as }l\rightarrow\infty
\text{.}\label{Eq_RetEquivHit}%
\end{equation}
Then $(A_{l})$ satisfies a LLT for return times in that, for every $\delta>0
$,
\begin{equation}
\mu_{A_{l}}(\varphi_{A_{l}}=k_{l})\sim\theta^{2}e^{-\theta\mu(A_{l})k_{l}%
}\,\mu(A_{l})\text{ \quad}%
\begin{array}
[c]{c}%
\text{as }l\rightarrow\infty\text{, uniformly}\\
\text{in }(k_{l})\in\mathcal{K}_{\delta}(A_{l})\text{.}%
\end{array}
\label{Eq_LLT_Ret_Abstr0}%
\end{equation}
\textbf{(b)} Condition (\ref{Eq_RetEquivHit}) is satisfied whenever there are
integers $\tau_{l}\geq0$ for which%
\begin{equation}
\mu(A_{l})\,\tau_{l}\longrightarrow0\text{ \quad as }l\rightarrow
\infty\text{,}\label{Eq_OrderOfTheTauL}%
\end{equation}
such that for all sequences $(k_{l})\in\mathcal{K}(A_{l})$,%
\begin{equation}
\left\vert \mu_{A_{l}^{\circ}}(T^{-\tau_{l}}\{\varphi_{A_{l}}=k_{l}%
\})-\mu(\varphi_{A_{l}}=k_{l})\right\vert =o(\mu(A_{l}))\text{ \quad as
}l\rightarrow\infty\text{,}\label{Eq_SailingAfterTau}%
\end{equation}
and%
\begin{equation}
\mu_{A_{l}^{\circ}}(\{\varphi_{A_{l}}\leq\tau_{l}\}\cap T^{-k_{l}}A_{l}%
)=o(\mu(A_{l}))\text{ \quad as }l\rightarrow\infty\text{.}%
\label{Eq_NothingHappensTooQuickly}%
\end{equation}

\end{theorem}

\begin{remark}
\label{Rem_Afterthought1}\textbf{a)} The most important case of the result is
that of a normalized exponential law, with $\theta=1$ and $A_{l}^{\bullet
}=\varnothing$ (so that there is nothing to check in (\ref{Eq_ReturningPart}),
and we simply have $\mu_{A_{l}^{\circ}}=\mu_{A_{l}}$).\newline\textbf{b)} The
second condition in (\ref{Eq_ReturningPart}) is trivially fulfilled if the
$\varphi_{A_{l}}$ are uniformly bounded on the sets $A_{l}^{\bullet}$, so that
$\sup_{l\geq1}\parallel1_{A_{l}^{\bullet}}\varphi_{A_{l}}\parallel_{\infty
}<\infty$, which will be the case in our applications of the theorem.\newline%
\textbf{c)} Note that (\ref{Eq_OrderOfTheTauL}) guarantees $k_{l}-\tau_{l}\sim
k_{l}$, so that (by Theorem \ref{T_AutoLLT_Hit}) condition
(\ref{Eq_SailingAfterTau}) is equivalent to $\left\vert \mu_{A_{l}^{\circ}%
}(T^{-\tau_{l}}\{\varphi_{A_{l}}=k_{l}-\tau_{l}\})-\mu(\varphi_{A_{l}}%
=k_{l})\right\vert =o(\mu(A_{l}))$.\newline\textbf{d)} Observe that neither
part of the theorem presupposes that the conditions should hold uniformly in
$(k_{l})$.\newline
\end{remark}

The crucial property (\ref{Eq_RetEquivHit}) of the conditional measures
$\mu_{A_{l}^{\circ}}$ is similar to the basic property (\ref{Eq_LLTgoodBaby}),
and it readily implies (\ref{Eq_LLT_Ret_Abstr0}). Part b) of the Theorem is
more substantial, and it will allows us to show that (\ref{Eq_RetEquivHit}) is
in fact satisfied in various interesting situations.\newline

We will actually establish a stronger version of the theorem, which asserts
uniform asymptotics $\nu_{l}(\varphi_{A_{l}}=k_{l})\sim\theta^{2}e^{-\theta
\mu(A_{l})k_{l}}\,\mu(A_{l})$\ for suitable sequences $(\nu_{l})$ of measures.
This variant will be important for the results of the next subsection. If, for
$l\geq1$, $\mathcal{V}_{l}$ is a family of probabilites on $(A_{l}%
,\mathcal{A}\cap A_{l})$, an element $(\nu_{l})$ of the product space $%
%TCIMACRO{\tprod \nolimits_{l\geq1}}%
%BeginExpansion
{\textstyle\prod\nolimits_{l\geq1}}
%EndExpansion
\mathcal{V}_{l}$ is a sequence of probability measures $\nu_{l}$ with $\nu
_{l}\in\mathcal{V}_{l}$ for every $l\geq1$. Given a measure $\nu_{l}$ we let
$\nu_{l,B}$ denote $\nu_{l}$ conditioned on $B\in\mathcal{A}$ so that
$\nu_{l,B}(E)=\nu_{l}(B)^{-1}\nu_{l}(B\cap E)$, $E\in\mathcal{A}$. We will prove

\begin{theorem}
[\textbf{LLT for exponential return times; extended version}]%
\label{T_LLT_RetGeneral}Let $T$ be an ergodic measure-preserving map on the
probability space $(X,\mathcal{A},\mu)$, $\theta\in(0,1]$, and $(A_{l}%
)_{l\geq1}$ a sequence in $\mathcal{A}$ such that $0<\mu(A_{l})\rightarrow0$
and $A_{l}=A_{l}^{\bullet}\cup A_{l}^{\circ}$ (measurable and disjoint).
Suppose that the normalized hitting times converge to an exponential law,
\begin{equation}
\mu(\mu(A_{l})\varphi_{A_{l}}\leq t)\Longrightarrow1-e^{-\theta t}\text{ \quad
as }l\rightarrow\infty\text{.}\label{Eq_AssmExpoHTS}%
\end{equation}
\textbf{(a)} For $l\geq1$ let $\mathcal{V}_{l}$ be a family of probabilites on
$(A_{l},\mathcal{A}\cap A_{l})$, and assume that
\begin{equation}
\sup_{\nu_{l}\in\mathcal{V}_{l}}\left\vert \nu_{l}(A_{l}^{\circ}%
)-\theta\right\vert \longrightarrow0\text{ \quad as }l\rightarrow
\infty\text{,}\label{Eq_ReturningPartGeneral1}%
\end{equation}
while
\begin{equation}
\mu(A_{l})\,\sup_{\nu_{l}\in\mathcal{V}_{l}}\parallel1_{A_{l}^{\bullet}%
}\varphi_{A_{l}}\parallel_{L_{\infty}(\nu_{l}^{\bullet})}\longrightarrow
0\text{ \quad as }l\rightarrow\infty\text{.}\label{Eq_ReturningPartGeneral2}%
\end{equation}
Here we let $\nu_{l}^{\bullet}:=\nu_{l,A_{l}^{\bullet}}$ (the normalized
restriction of $\nu_{l}$ to $A_{l}^{\bullet}$) if $\nu_{l}(A_{l}^{\bullet}%
)>0$, and $\nu_{l}^{\bullet}:=0$ otherwise (and analogously for $\nu
_{l}^{\circ}$). If for every sequence $(k_{l})\in\mathcal{K}(A_{l})$,
\begin{equation}
\left\vert \nu_{l}^{\circ}(\varphi_{A_{l}}=k_{l})-\mu(\varphi_{A_{l}}%
=k_{l})\right\vert =o(\mu(A_{l}))\text{ \quad}%
\begin{array}
[c]{c}%
\text{as }l\rightarrow\infty\text{, uniformly}\\
\text{in }(\nu_{l})\in%
%TCIMACRO{\tprod \nolimits_{l\geq1}}%
%BeginExpansion
{\textstyle\prod\nolimits_{l\geq1}}
%EndExpansion
\mathcal{V}_{l}\text{,}%
\end{array}
\label{Eq_RetEquivHitGeneral}%
\end{equation}
then, for every $\delta>0$,%
\begin{equation}
\nu_{l}(\varphi_{A_{l}}=k_{l})\sim\theta^{2}e^{-\theta\mu(A_{l})k_{l}}%
\,\mu(A_{l})\text{ \quad}%
\begin{array}
[c]{c}%
\text{as }l\rightarrow\infty\text{, uniformly in}\\
(\nu_{l})\in%
%TCIMACRO{\tprod \nolimits_{l\geq1}}%
%BeginExpansion
{\textstyle\prod\nolimits_{l\geq1}}
%EndExpansion
\mathcal{V}_{l}\text{ and }(k_{l})\in\mathcal{K}_{\delta}(A_{l})\text{.}%
\end{array}
\label{Eq_LLT_Ret_AbstrGeneral}%
\end{equation}
\textbf{(b)} Condition (\ref{Eq_RetEquivHitGeneral}) is satisfied whenever
there are measurable $\tau_{l}:A_{l}^{\circ}\rightarrow\overline{\mathbb{N}%
}_{0}$, $l\geq1$, such that, for every $\varepsilon>0$,%
\begin{equation}
\nu_{l}^{\circ}(\,\mu(A_{l})\tau_{l}>\varepsilon)=o(\mu(A_{l}))\text{ \quad}%
\begin{array}
[c]{c}%
\text{as }l\rightarrow\infty\text{, uniformly}\\
\text{in }(\nu_{l})\in%
%TCIMACRO{\tprod \nolimits_{l\geq1}}%
%BeginExpansion
{\textstyle\prod\nolimits_{l\geq1}}
%EndExpansion
\mathcal{V}_{l}\text{.}%
\end{array}
\label{Eq_OrderOfTheTauLGeneral1}%
\end{equation}
and some family $\mathcal{V}$ of probability measures on $(X,\mathcal{A})$
which is LLT-uniform for $(A_{l})$ and such that
\begin{equation}
\nu_{l,\{\tau_{l}=j\}}^{\circ}\circ T^{-j}\in\mathcal{V}\text{ \quad}%
\begin{array}
[c]{c}%
\text{for }l\geq1\text{, }\nu_{l}\in\mathcal{V}_{l}\text{, and those}\\
j\text{ for which }\nu_{l}^{\circ}(\tau_{l}=j)>0\text{,}%
\end{array}
\label{Eq_SailingAfterTauGeneral}%
\end{equation}
while for all sequences $(k_{l})\in\mathcal{K}(A_{l})$,%
\begin{equation}
\nu_{l}^{\circ}(\{\varphi_{A_{l}}\leq\tau_{l}\}\cap T^{-k_{l}}A_{l}%
)=o(\mu(A_{l}))\text{ \quad}%
\begin{array}
[c]{c}%
\text{as }l\rightarrow\infty\text{, uniformly}\\
\text{in }(\nu_{l})\in%
%TCIMACRO{\tprod \nolimits_{l\geq1}}%
%BeginExpansion
{\textstyle\prod\nolimits_{l\geq1}}
%EndExpansion
\mathcal{V}_{l}\text{.}%
\end{array}
\label{Eq_NothingHappensTooQuicklyGeneral}%
\end{equation}

\end{theorem}

\begin{remark}
\textbf{a)} Taking $\mathcal{V}_{l}:=\{\mu_{A_{l}}\}$ and constant functions
$\tau_{l}$, we see that this result contains Theorem \ref{T_LLT_Ret0}.
\newline\textbf{b)} The proof shows that (\ref{Eq_ReturningPartGeneral2}) can
be replaced by the assertion that for all sequences $(k_{l})\in\mathcal{K}%
(A_{l})$ we have $\nu_{l}^{\bullet}(\varphi_{A_{l}}=k_{l})=o(1/\mu(A_{l}))$
uniformly in $(\nu_{l})\in%
%TCIMACRO{\tprod \nolimits_{l\geq1}}%
%BeginExpansion
{\textstyle\prod\nolimits_{l\geq1}}
%EndExpansion
\mathcal{V}_{l}$.\newline\textbf{c)} Also, as can be seen from the proof,
(\ref{Eq_OrderOfTheTauLGeneral1}) can be replaced by the more generous
condition that for every $(k_{l})\in\mathcal{K}(A_{l})$ and $\varepsilon>0$,%
\begin{gather}
\nu_{l}^{\circ}(\,\mu(A_{l})\tau_{l}>\varepsilon)\longrightarrow0\text{ \quad
and \quad}\nu_{l}^{\circ}(\,\{\mu(A_{l})\tau_{l}>\varepsilon\}\cap
\{\varphi_{A_{l}}=k_{l}\})=o(\mu(A_{l}))\nonumber\\
\text{as }l\rightarrow\infty\text{, uniformly in }(\nu_{l})\in%
%TCIMACRO{\tprod \nolimits_{l\geq1}}%
%BeginExpansion
{\textstyle\prod\nolimits_{l\geq1}}
%EndExpansion
\mathcal{V}_{l}\text{.}\label{Eq_OrderOfTheTauLGeneral5}%
\end{gather}

\end{remark}%

%TCIMACRO{\TeXButton{VSs}{\vspace{0.3cm}}}%
%BeginExpansion
\vspace{0.3cm}%
%EndExpansion
%

%TCIMACRO{\TeXButton{NI}{\noindent}}%
%BeginExpansion
\noindent
%EndExpansion
\textbf{LLT for consecutive asymptotically exponential return and hitting
times. Conditions on positions of the hits.} Recall that $T_{A}:X\rightarrow
A$ denotes the first entrance map of a set $A$. Having concentrated on the
\emph{first} return- or hitting time $\varphi_{A}$ of a suitable small set $A$
so far, we now turn to the full sequence $\Phi_{A}:=(\varphi_{A},\varphi
_{A}\circ T_{A},\varphi_{A}\circ T_{A}^{2},\ldots)$ of consecutive waiting
times between successive visits to $A$. Distributional convergence of the
normalized sequences $\mu(A_{l})\Phi_{A_{l}}$ to an independent sequence
$(\mathcal{E}^{(0)},\mathcal{E}^{(1)},\ldots)$ of normalized exponential
random variables is equivalent to saying that the associated \emph{normalized
counting processes} $\mathrm{N}_{A_{l}}:X\rightarrow\mathbb{D}[0,\infty)$
given by $\mathrm{N}_{A_{l}}=(\mathrm{N}_{A_{l},t})_{t\geq0}$ with
$\mathrm{N}_{A_{l},t}:=%
%TCIMACRO{\tsum _{k=1}^{\left\lfloor t/\mu(A_{l})\right\rfloor }}%
%BeginExpansion
{\textstyle\sum_{k=1}^{\left\lfloor t/\mu(A_{l})\right\rfloor }}
%EndExpansion
1_{A_{l}}\circ T^{k}$, converge to a standard Poisson counting process
$\mathrm{N}=(\mathrm{N}_{t})_{t\geq0}$ (using the Skorohod $\mathcal{J}_{1}%
$-topology on the space $\mathbb{D}[0,\infty)$ of cadlag functions
$g:[0,\infty)\rightarrow\mathbb{R}$). \emph{As a caveat we recall that having
an exponential limit law for the }first\emph{\ return times }$\varphi_{A_{l}}%
$\emph{\ does not automatically imply the latter behaviour}. Indeed, every
stationary sequence of exponential variables can occur in the limit. (This is
a special case of Theorem 2.1 in \cite{Z11}). Asymptotic independence thus
requires extra assumptions. Statements a) and b) of the next theorem provide
sufficient conditions for a local version of Poisson asymptotics expressed in
terms of the waiting time processes $\Phi_{A_{l}}$.

In parts a') and b') we further improve this result by conditioning on hitting
a specific part $C$ of a target set $A$. In the well-studied setup of
distributional convergence, such spatiotemporal limit theorems for dynamical
systems have been studied recently in \cite{PS2} and \cite{ZWhenAndWhere}.
Here we establish a local version of these refined results.

\begin{theorem}
[\textbf{LLT for consecutive times plus positions}]%
\label{T_LLT_ConsecutiveTimes}Under the assumptions of Theorem
\ref{T_LLT_RetGeneral} a) the following hold.\newline\newline\textbf{a)}
Suppose that for every $(k_{l})\in\mathcal{K}(A_{l})$,%
\begin{equation}
\nu_{l,\{\varphi_{A_{l}}=k_{l}\}}\circ T_{A_{l}}^{-1}\in\mathcal{V}_{l}\text{
\qquad}%
\begin{array}
[c]{c}%
\text{whenever }l\geq1\text{, }\nu_{l}\in\mathcal{V}_{l}\text{,}\\
\text{and }\nu_{l}(\varphi_{A_{l}}=k_{l})>0\text{.}%
\end{array}
\label{Eq_InvarianceVls}%
\end{equation}
Then, for every $d\geq1$ and $\delta>0$, we have as $l\rightarrow\infty$,
\begin{gather}
\nu_{l}\left(  \bigcap_{j=1}^{d}T_{A_{l}}^{-(j-1)}\{\varphi_{A_{l}}%
=k_{l}^{(j)}\}\right)  \sim\,\frac{\theta^{2d}\mu(A_{l})^{d}}{e^{\theta
\mu(A_{l})\sum\nolimits_{j=1}^{d}k_{l}^{(j)}}}\tag{$\clubsuit_d$}\\
\text{as }l\rightarrow\infty\text{, uniformly in }(\nu_{l})\in%
%TCIMACRO{\tprod \nolimits_{l\geq1}}%
%BeginExpansion
{\textstyle\prod\nolimits_{l\geq1}}
%EndExpansion
\mathcal{V}_{l}\text{ and }(k_{l}^{(1)}),\ldots,(k_{l}^{(d)})\in
\mathcal{K}(A_{l})\text{.}\nonumber
\end{gather}
\textbf{a')} Suppose, moreover, that $(C_{l}^{(1)}),\ldots,(C_{l}^{(d)})$ and
$(c_{l}^{(1)}),\ldots,(c_{l}^{(d)})$ are sequences in $\mathcal{A}\cap A_{l}$
and $(0,1]$, respectively, such that%
\begin{gather}
\nu_{l,C_{l}^{(j)}}\in\mathcal{V}_{l}\qquad\text{for }l\geq1\text{ and }%
\nu_{l}\in\mathcal{V}_{l},\text{ \qquad while}\label{Eq_kysdjhfbhjeabgfhjbevb}%
\\
\nu_{l,\{\varphi_{A_{l}}=k_{l}\}}\circ T_{A_{l}}^{-1}(C_{l}^{(j)})\sim
c_{l}^{(j)}\text{\quad as }l\rightarrow\infty\text{, uniformly in }(\nu
_{l})\in%
%TCIMACRO{\tprod \nolimits_{l\geq1}}%
%BeginExpansion
{\textstyle\prod\nolimits_{l\geq1}}
%EndExpansion
\mathcal{V}_{l}\text{.}\label{Eq_AsyNuCl}%
\end{gather}
Then,%
\begin{gather}
\nu_{l}\left(  \bigcap_{j=1}^{d}T_{A_{l}}^{-(j-1)}\left(  \{\varphi_{A_{l}%
}=k_{l}^{(j)}\}\cap T_{A_{l}}^{-1}C_{l}^{(j)}\right)  \right)  \sim
\frac{\,\theta^{2d}\mu(A_{l})^{d}\prod_{j=1}^{d}c_{l}^{(j)}}{e^{\theta
\mu(A_{l})\sum\nolimits_{j=1}^{d}k_{l}^{(j)}}}\tag{$\blacklozenge_d$}\\
\text{uniformly in }(\nu_{l})\in%
%TCIMACRO{\tprod \nolimits_{l\geq1}}%
%BeginExpansion
{\textstyle\prod\nolimits_{l\geq1}}
%EndExpansion
\mathcal{V}_{l}\text{ and }(k_{l}^{(1)}),\ldots,(k_{l}^{(d)})\in
\mathcal{K}(A_{l})\text{.}\nonumber
\end{gather}
\textbf{b)} In the situation of a), if $\mathcal{V}$ is LLT-uniform for
$(A_{l})$ with%
\begin{equation}
\nu_{\{\varphi_{A_{l}}=k_{l}\}}\circ T_{A_{l}}^{-1}\in\mathcal{V}_{l}%
\qquad\text{for }l\geq1\text{ and }\nu\in\mathcal{V}\text{,}%
\label{Eq_Hourglass0}%
\end{equation}
then, for every $d\geq1$ and $\delta>0$, we have as $l\rightarrow\infty$,%
\begin{gather}
\nu\left(  \bigcap_{j=1}^{d}T_{A_{l}}^{-(j-1)}\{\varphi_{A_{l}}=k_{l}%
^{(j)}\}\right)  \sim\,\frac{\theta^{2d-1}\mu(A_{l})^{d}}{e^{\theta\mu
(A_{l})\sum\nolimits_{j=1}^{d}k_{l}^{(j)}}}\tag{$\spadesuit_d$}\\
\text{uniformly in }\nu\in\mathcal{V}\text{ and }(k_{l}^{(1)}),\ldots
,(k_{l}^{(d)})\in\mathcal{K}(A_{l})\text{.}\nonumber
\end{gather}
\textbf{b')} Under the assumptions of a') and b), if%
\begin{equation}
\nu_{\{\varphi_{A_{l}}=k_{l}\}}\circ T_{A_{l}}^{-1}(C_{l}^{(1)})\sim
c_{l}^{(1)}\text{\quad as }l\rightarrow\infty\text{, uniformly in }\nu
\in\mathcal{V}\text{.}\label{Eq_Hourglass1}%
\end{equation}
then
\begin{gather}
\nu\left(  \bigcap_{j=1}^{d}T_{A_{l}}^{-(j-1)}\left(  \{\varphi_{A_{l}}%
=k_{l}^{(j)}\}\cap T_{A_{l}}^{-1}C_{l}^{(j)}\right)  \right)  \sim
\,\frac{\theta^{2d-1}\mu(A_{l})^{d}\prod_{j=1}^{d}c_{l}^{(j)}}{e^{\theta
\mu(A_{l})\sum\nolimits_{j=1}^{d}k_{l}^{(j)}}}\tag{$\lozenge_d$}\\
\text{uniformly in }\nu\in\mathcal{V}\text{ and }(k_{l}^{(1)}),\ldots
,(k_{l}^{(d)})\in\mathcal{K}(A_{l})\text{.}\nonumber
\end{gather}

\end{theorem}%

%TCIMACRO{\TeXButton{VSs}{\vspace{0.3cm}}}%
%BeginExpansion
\vspace{0.3cm}%
%EndExpansion

\begin{remark}
There is no restriction on the asymptotics of $c_{l}^{(j)}$ as $l\rightarrow
\infty$.
\end{remark}

The next section discusses some natural situations in which this result applies.

\section{Local limit theorems in some concrete
situations\label{Sec_ConcreteLLTs}}

We illustrate the use of the abstract results above by considering some basic
piecewise invertible dynamical systems.%

%TCIMACRO{\TeXButton{VSs}{\vspace{0.3cm}}}%
%BeginExpansion
\vspace{0.3cm}%
%EndExpansion
%

%TCIMACRO{\TeXButton{NI}{\noindent}}%
%BeginExpansion
\noindent
%EndExpansion
\textbf{Piecewise invertible systems.} We consider situations in which
$(X,\mathrm{d}_{X})$ is a metric space with $\mathrm{diam}(X)<\infty$,
equipped with its Borel $\sigma$-algebra $\mathcal{A}=\mathcal{B}_{X}$. A
\emph{piecewise invertible probability preserving system} on $X$ is a
quintuple $\mathfrak{S}=(X,\mathcal{A},\mu,T,\xi)$, where $T:X\rightarrow X$
preserves the probability measure $\mu$ on $\mathcal{A}$, $\xi=\xi_{1}$ is a
(finite or) countable partition mod $\mu$ of $X$ into pairwise disjoint open
sets such that each \emph{branch} of $T$, i.e. its restriction to any of its
\emph{cylinders} $Z\in\xi$ is a homeomorphism onto $TZ$ (automatically
\emph{null-preserving} with respect to $\mu$, meaning that $\mu\mid_{Z}\circ
T^{-1}\ll\mu$). The system is called \emph{uniformly expanding} if there is
some $\rho\in(0,1)$ such that $\mathrm{d}_{X}(x,y)\leq\rho\cdot\mathrm{d}%
_{X}(Tx,Ty)$ whenever $x,y\in Z\in\xi$.

We let $\xi_{n}$ denote the family of \emph{cylinders of rank }$n$, that is,
the sets of the form $Z=[Z_{0},\ldots,Z_{n-1}]:=\bigcap_{i=0}^{n-1}T^{-i}%
Z_{i}$ with $Z_{i}\in\xi$. Write $\xi_{n}(x)$ for the element of $\xi_{n}$
containing $x$ (which is defined a.e.). Each \emph{iterate} $\mathfrak{S}%
^{n}=(X,\mathcal{A},\mu,T^{n},\xi_{n})$, $n\geq1$, of the system is again
piecewise invertible. The \emph{inverse branches} of $T$ and its iterates will
be denoted $v_{Z}:=(T^{n}\mid_{Z})^{-1}:T^{n}Z\rightarrow Z$, $Z\in\xi_{n}$.
As the branches are null-preserving, all $v_{Z}$ possess Radon-Nikodym
derivatives $v_{Z}^{\prime}:=d(\mu\circ v_{Z})/d\mu$. Some of the examples are
\emph{interval maps}, meaning that the $Z\in\xi$ are subintervals of
$\mathbb{R}$.

The system is \emph{Markov} if $\mu(TZ\cap Z^{\prime})>0$ for $Z,Z^{\prime}%
\in\xi$ implies $Z^{\prime}\subseteq TZ$, so that $TZ$ is $\xi$-measurable
(mod $\mu$). In this case each iterate $\mathfrak{S}^{n}$ is Markov as well,
and in fact $T^{n}Z$ is $\xi$-measurable (mod $\mu$) whenever $n\geq1$ and
$Z\in\xi_{n}$. Moreover,%
\begin{equation}
T^{k}C\text{ is }\xi_{l}\text{-measurable (mod }\mu\text{) for }k,l\geq1\text{
and }C\in\xi_{k+l}\text{. }\label{Eq_MbilityOfImages}%
\end{equation}
(Indeed, $C=B\cap T^{-k}A$ with $B\in\xi_{k}$, $A\in\xi_{l}$ and $T^{k}%
C=T^{k}B\cap A\in\{\varnothing,A\}$.)%

%TCIMACRO{\TeXButton{VSs}{\vspace{0.3cm}}}%
%BeginExpansion
\vspace{0.3cm}%
%EndExpansion
%

%TCIMACRO{\TeXButton{NI}{\noindent}}%
%BeginExpansion
\noindent
%EndExpansion
\textbf{Gibbs-Markov maps.} One important basic class of such Markov systems
$\mathfrak{S}=(X,\mathcal{A},\mu,T,\xi)$ is that of probability preserving
\emph{Gibbs-Markov maps} (\emph{GM maps}). This means that the system has a
uniformly expanding iterate $T^{N}$, and satisfies the \emph{big image
property}, in that $\flat:=\inf_{Z\in\xi}\mu(TZ)>0$. Moreover, the
$v_{Z}^{\prime}$, $Z\in\xi$, are required to have versions which are well
behaved in that there exists some $r>0$ such that $\left\vert v_{Z}^{\prime
}(x)/v_{Z}^{\prime}(y)-1\right\vert \leq r\,\mathrm{d}_{X}(x,y)$ whenever
$x,y\in TZ$, $Z\in\xi$ (see \cite{AD}). Equivalently, letting $\beta$ denote
the partition generated by $T\xi$, and $\mathrm{R}_{\beta}(w):=\sup_{B\in
\beta}\mathrm{R}_{B}(w)$ with $\mathrm{R}_{B}(w):=\inf\{s>0:w(x)\leq
(1+s\,\mathrm{d}_{X}(x,y))w(y)$ for $x,y\in B\}$ the \emph{regularity }of a
nonnegative function\emph{\ }$w$\emph{\ }on\emph{\ }$B$, we require that
$\widehat{T}\left(  \mu(Z)^{-1}1_{Z}\right)  =\mu(Z)^{-1}v_{Z}^{\prime}%
\in\mathcal{U}(r)$, where $\mathcal{U}(s):=\{u\in\mathcal{D}(\mu
):\mathrm{R}_{\beta}(u)\leq s\}$. In this case $r=r(\mathfrak{S})$ can (and
will) be chosen in such a way that in fact
\begin{equation}
\left\vert \frac{v_{Z}^{\prime}(x)}{v_{Z}^{\prime}(y)}-1\right\vert \leq
r\,\mathrm{d}_{X}(x,y)\text{ \quad whenever }n\geq1\text{ and }x,y\in
T^{n}Z,Z\in\xi_{n}\text{,}\label{Eq_RegularityForGMSys}%
\end{equation}
and $v_{Z}^{\prime}$ will always denote such versions of the a.e. defined
Radon-Nikodym derivatives $d(\mu\circ v_{Z})/d\mu$.%

%TCIMACRO{\TeXButton{VSs}{\vspace{0.3cm}}}%
%BeginExpansion
\vspace{0.3cm}%
%EndExpansion

To rule out trivial obstacles to LLTs for return times caused by periodicities
of the system, we focus on \emph{mixing} GM-maps. For a first concrete variant
of our LLTs we consider the basic scenario of cylinders of increasing rank
which shrink to some given point $x^{\ast}$.

\begin{theorem}
[\textbf{LLT for shrinking cylinders of GM maps}]\label{T_LLTGMCyls1}Let
$(X,\mathcal{A},\mu,T,\xi)$ be a mixing probability preserving Gibbs-Markov
system. Let $x^{\ast}\in X$ be a point such that the cylinder $A_{l}:=\xi
_{l}(x^{\ast})$ is defined for every $l\geq1$. Fix any $r>0$ and set
$\mathcal{V}:=\{u\odot\mu:u\in\mathcal{U}(r)\}$ and $\mathcal{V}_{l}%
:=\{u\odot\mu:u\in\mathcal{U}_{l}\}$ with $\mathcal{U}_{l}=\mathcal{U}%
_{l}(r):=\{u\in\mathcal{D}(\mu):u=1_{A_{l}}u$ and $\mathrm{R}_{A_{l}}(u)\leq
r\}$.\newline\newline\textbf{A)} If $x^{\ast}$ is not periodic, let
$\theta:=1$. Then for any $d\geq1$ and $(k_{l}^{(1)}),\ldots,(k_{l}^{(d)}%
)\in\mathcal{K}(A_{l})$, statements ($\spadesuit_{d}$) and ($\clubsuit_{d}$)
hold.\newline\newline\textbf{B)} If $x^{\ast}$ is periodic, $T^{p}x^{\ast
}=x^{\ast}$ with $p\geq1$ minimal and $\theta:=1-v_{A_{p}}^{\prime}(x^{\ast
})\in(0,1)$, then for any $d\geq1$ and $(k_{l}^{(1)}),\ldots,(k_{l}^{(d)}%
)\in\mathcal{K}(A_{l})$, statements ($\spadesuit_{d}$) and ($\clubsuit_{d}$) hold.
\end{theorem}

Another natural situation is that of small sets which consist of rank-one cylinders.

\begin{theorem}
[\textbf{Spatiotemporal LLT for unions of cylinders of GM maps}]%
\label{T_LLTGMCyls2}Let $(X,\mathcal{A},\mu,T,\xi)$ be a mixing probability
preserving Gibbs-Markov system, and $(A_{l})$ a sequence of asymptotically
rare events such that each $A_{l}$ is $\xi$-measurable. Set $\theta:=1$, fix
$r>0$ and define $\mathcal{V}:=\{(%
%TCIMACRO{\tint \nolimits_{C}}%
%BeginExpansion
{\textstyle\int\nolimits_{C}}
%EndExpansion
u\,d\mu)^{-1}1_{C}\,u\odot\mu:u\in\mathcal{U}(r)$, $C\subseteq X$ a $\xi
$-measurable set, $%
%TCIMACRO{\tint \nolimits_{C}}%
%BeginExpansion
{\textstyle\int\nolimits_{C}}
%EndExpansion
u\,d\mu>0\}$ and $\mathcal{V}_{l}:=\{(%
%TCIMACRO{\tint \nolimits_{C}}%
%BeginExpansion
{\textstyle\int\nolimits_{C}}
%EndExpansion
u\,d\mu)^{-1}1_{C}\,u\odot\mu:u\in\mathcal{U}(r),C\subseteq A_{l}$ $\xi
$-measurable, $%
%TCIMACRO{\tint \nolimits_{C}}%
%BeginExpansion
{\textstyle\int\nolimits_{C}}
%EndExpansion
u\,d\mu>0\}$.\newline\newline\textbf{A)} Let $d\geq1$ and $(k_{l}%
^{(1)}),\ldots,(k_{l}^{(d)})\in\mathcal{K}(A_{l}),$ then statements
($\spadesuit_{d}$) and ($\clubsuit_{d} $) hold.\newline\newline\textbf{B)}
Suppose, in addition, that $\mathrm{diam}(A_{l})\rightarrow0$ as
$l\rightarrow\infty$, that $A_{l}\subseteq A_{l}^{\prime}\ $for suitable
$A_{l}^{\prime}\in\beta$, and that $(C_{l}^{(1)}),\ldots,(C_{l}^{(d)})$ are
sequences of $\xi$-measurable sets with $C_{l}^{(j)}\subseteq A_{l}$. Then,
statements ($\lozenge_{d}$) and ($\blacklozenge_{d}$) hold with $c_{l}%
^{(j)}:=\mu_{A_{l}}(C_{l}^{(j)})$.
\end{theorem}%

%TCIMACRO{\TeXButton{VSs}{\vspace{0.3cm}}}%
%BeginExpansion
\vspace{0.3cm}%
%EndExpansion

\section{Proofs for Section \ref{Sec_AbstractLLTs}}%

%TCIMACRO{\TeXButton{NI}{\noindent}}%
%BeginExpansion
\noindent
%EndExpansion
\textbf{More background material.} A fundamental result on return- and hitting
time distributions of small sets is the following relation between their
respective limiting behaviours, see \cite{HaydnLacVai05} or \cite{AS}.

\begin{theorem}
[\textbf{Hitting-time statistics versus return-time statistics}]%
\label{T_HittingTimeVsReturnTimeLimits}Let $(X,\mathcal{A},\mu,T)$ be an
ergodic probability-preserving system, and $(A_{l})_{l\geq1}$ a sequence of
asymptotically rare events. Then
\begin{equation}
\mu(\mu(A_{l})\,\varphi_{A_{l}}\leq t)\Longrightarrow F(t)\text{ \quad for
some sub-probability }F\text{ on }[0,\infty)\label{Eq_HitHittingTimeCge}%
\end{equation}
iff
\begin{equation}
\mu_{A_{l}}(\mu(A_{l})\,\varphi_{A_{l}}\leq t)\Longrightarrow\widetilde
{F}(t)\text{ \quad for some sub-probability }\widetilde{F}\text{ on }%
[0,\infty)\text{.}\label{Eq_RetReturnTimeCge}%
\end{equation}
In this case $F$ and $\widetilde{F}$ satisfy \newline%
\begin{equation}%
%TCIMACRO{\tint _{0}^{t}}%
%BeginExpansion
{\textstyle\int_{0}^{t}}
%EndExpansion
(1-\widetilde{F}(s))\,ds=F(t)\text{\ \quad for }t\geq0\text{.}%
\label{Eq_RelationBetweenHittingAndReturnDistrFun}%
\end{equation}
Through this integral equation each of $F$ and $\widetilde{F}$ uniquely
determines the other. Moreover, $F$ is necessarily continuous and concave with
$F(t)\leq t$, while $\widetilde{F}$ is always a probability distribution
function satisfying $%
%TCIMACRO{\tint _{0}^{\infty}}%
%BeginExpansion
{\textstyle\int_{0}^{\infty}}
%EndExpansion
(1-\widetilde{F}(s))\,ds\leq1$.
\end{theorem}

\begin{remark}
\label{Rem_PropF}This result explains why any limit $F$ in
(\ref{Eq_HitHittingTimeCge}) has a density with a non-increasing cadlag
version $F^{\prime}$, as we can simply take $F^{\prime}(t):=1-\widetilde
{F}(t)$. It also shows that $0\leq F^{\prime}\leq1$.
\end{remark}%

%TCIMACRO{\TeXButton{VSs}{\vspace{0.3cm}}}%
%BeginExpansion
\vspace{0.3cm}%
%EndExpansion
%

%TCIMACRO{\TeXButton{NI}{\noindent}}%
%BeginExpansion
\noindent
%EndExpansion
\textbf{Abstract LLTs for hitting and return times.} We can now turn to the

\begin{proof}
[\textbf{Proof of Theorem \ref{T_LLT_Hit}}]\textbf{(a)} The key to this
statement is the standard fact (from the theory of induced transformations)
that, for any $A\in\mathcal{A}$ of positive measure,
\begin{equation}
\mu(\varphi_{A}=k)=\mu(A\cap\{\varphi_{A}\geq k\})\text{ \quad for }%
k\geq1\text{.}\label{Eq_MeasureInducing}%
\end{equation}
(Since $\{\varphi_{A}>k\}=T^{-1}(A^{c}\cap\{\varphi_{A}\geq k\})$ and hence
$\mu(\varphi_{A}>k)=\mu(\varphi_{A}\geq k)-\mu(A\cap\{\varphi_{A}\geq k\})$.)
Recalling Theorem \ref{T_HittingTimeVsReturnTimeLimits}, we see that our
assumption (\ref{Eq_DLT_HTS}) implies (\ref{Eq_RetReturnTimeCge}) with $F$ and
$\widetilde{F}$ related by (\ref{Eq_RelationBetweenHittingAndReturnDistrFun}).
In view of the latter, if $t>0 $ is a continuity point of $F^{\prime}$, then
it is also a continuity point of the probability distribution function
$\widetilde{F}$, and $F^{\prime}(t)=1-\widetilde{F}(t)$. By elementary
arguments, (\ref{Eq_RetReturnTimeCge}) implies that for any sequence $(k_{l})$
in $\mathbb{N} $ with $\mu(A_{l})k_{l}\rightarrow t$ we have
\[
\mu_{A_{l}}(\mu(A_{l})\varphi_{A_{l}}\geq\mu(A_{l})k_{l})\rightarrow
1-\widetilde{F}(t)=F^{\prime}(t)\text{ \quad as }l\rightarrow\infty\text{.}%
\]
Now (\ref{Eq_LLT_HTS_Simple}) follows since (\ref{Eq_MeasureInducing}) shows
that
\begin{equation}
\mu(\varphi_{A_{l}}=k_{l})=\mu_{A_{l}}(\mu(A_{l})\varphi_{A_{l}}\geq\mu
(A_{l})k_{l})\,\mu(A_{l})\text{.}\label{Eq_Astep}%
\end{equation}
\textbf{(c)} Similarly, under the assumption of \textbf{(c)}, we see that
$F^{\prime}=1-\widetilde{F}$ on $I$, and this function is uniformly continuous
on any compact subset $J$ of $I$. It is then straightforward that
\[
\mu_{A_{l}}(\mu(A_{l})\varphi_{A_{l}}\geq t)\rightarrow F^{\prime}(t)\text{
\quad as }l\rightarrow\infty\text{, uniformly in }t\in J\text{,}%
\]
and a second look at (\ref{Eq_Astep}) proves (\ref{Eq_LLT_HTS_cpt1}%
).\newline\newline\textbf{(b)} To prove this converse to (a), note first that
if $A\in\mathcal{A}$, then
\begin{equation}
\mu(\varphi_{A}=k)\geq\mu(\varphi_{A}=k+1)\text{ \quad for }k\geq1\text{,}%
\end{equation}
because $\{\varphi_{A}=k+1\}\subseteq T^{-1}\{\varphi_{A}=k\}$ and $T$
preserves $\mu$. Therefore, (\ref{Eq_AssmLLTHitConverse}) implies that $G$ is
non-increasing on $D$. It thus extends to a monotone (hence measurable)
function $G:[0,\infty)\rightarrow\lbrack0,\infty)$, so that $F(t):=\int
_{0}^{t}G(s)\,ds$ is (uniquely) defined for all $t\geq0$.\newline

Below we show that any sequence of indices $l_{j}\nearrow\infty$ contains a
further subsequence of indices $l_{j_{i}}\nearrow\infty$ along which
\begin{equation}
\mu(\mu(A_{l_{j_{i}}})\varphi_{A_{l_{j_{i}}}}\leq t)\Longrightarrow F(t)\text{
\quad as }i\rightarrow\infty\text{.}%
\end{equation}
By a routine argument, this implies (\ref{Eq_DLT_HTS}).\newline\newline To
this end, start from $(l_{j})$ and appeal to the Helly selection principle to
obtain $l_{j_{i}}\nearrow\infty$ and a sub-probability distribution function
$H$ on $[0,\infty)$ such that $\mu(\mu(A_{l_{j_{i}}})\varphi_{A_{l_{j_{i}}}%
}\leq t)\Longrightarrow H(t)$ as $i\rightarrow\infty$. Apply Theorem
\ref{T_HittingTimeVsReturnTimeLimits} to $(A_{l_{j_{i}}})$ to see that
$H(0)=0=F(0)$, and that $H$ has a non-increasing density with right-continuous
version $H^{\prime}$ on $[0,\infty)$. Let $D_{\ast}:=\{t\in D:H^{\prime}$ and
$F^{\prime}$ are continuous at $t\}$, then $D_{\ast}$ has full Lebesgue
measure in $(0,\infty)$.

Now take any $t\in D_{\ast}$ and a sequence $(k_{l})_{l\geq1}$ of integers for
which $\mu(A_{l})k_{l}\rightarrow t$. Having already established statement
(a), we see that
\[
\mu(\varphi_{A_{l_{j_{i}}}}=k_{l_{j_{i}}})\sim H^{\prime}(t)\cdot
\mu(A_{l_{j_{i}}})\text{ \quad as }l\rightarrow\infty\text{,}%
\]
but due to assumption (\ref{Eq_AssmLLTHitConverse}) we also have $\mu
(\varphi_{A_{l_{j_{i}}}}=k_{l_{j_{i}}})\sim G(t)\cdot\mu(A_{l_{j_{i}}})$.
Therefore, $H^{\prime}(t)=G(t)=F^{\prime}(t)$ for $t\in D_{\ast}$, proving
that $H=F$.
\end{proof}%

%TCIMACRO{\TeXButton{VSs}{\vspace{0.3cm}}}%
%BeginExpansion
\vspace{0.3cm}%
%EndExpansion

We can then validate the asserted robustness property.

\begin{proof}
[\textbf{Proof of Theorem \ref{T_RobustLLThit}}]\textbf{(i)} Assume
(\ref{T_bchscbhdsbchsbch}), then Theorem \ref{T_LLT_Hit} (b) ensures
distributional convergence of hitting times, $\mu(\mu(A_{l})\varphi_{A_{l}%
}\leq t)\Longrightarrow F(t)$ as in (\ref{Eq_DLT_HTS}), with $F(t):=\int
_{0}^{t}G(s)\,ds$ and $G=F^{\prime}$. This, however, is robust under passage
to $(B_{l})$, meaning that $\mu(\mu(B_{l})\varphi_{B_{l}}\leq
t)\Longrightarrow F(t)$ (see below). Theorem \ref{T_LLT_Hit} (a) then gives
$\mu(\varphi_{B_{l}}=k_{l})\sim G(t)\cdot\mu(B_{l})$ whenever $\mu(A_{l}%
)k_{l}\rightarrow t\in D$, as required.\newline\newline\textbf{(ii)} To see
that indeed $\mu(\mu(B_{l})\varphi_{B_{l}}\leq t)\Longrightarrow F(t)$ we can,
for instance, invoke results from \cite{Z11} (using the notation introduced
there): Suppose otherwise, then there is a subsequence of indices
$l_{j}\nearrow\infty$, a continuity point $t_{\ast}>0$ of $F$, and some
$\vartheta>0$ such that $\mid\mu(\mu(B_{l_{j}})\varphi_{B_{l_{j}}}\leq
t_{\ast})-F(t_{\ast})\mid\geq\vartheta$ for $j\geq1$. A standard compactness
argument (Lemma 4.4 of \cite{Z11}) shows that there is a further subsequence
of indices $l_{j_{i}}\nearrow\infty$ along which not only the first hitting
time $\varphi_{A_{l}}$, but even the full sequence $\Phi_{A_{l}}%
=(\varphi_{A_{l}},\varphi_{A_{l}}\circ T_{A_{l}},\varphi_{A_{l}}\circ
T_{A_{l}}^{2},\ldots)$ of all consecutive hitting times converges in
distribution to some process $\Phi=(\varphi^{(0)},\varphi^{(1)},\varphi
^{(2)},\ldots)$, in that $\mu(A_{l_{j_{i}}})\Phi_{A_{l_{j_{i}}}}\overset{\mu
}{\Longrightarrow}\Phi$ (meaning convergence of finite-dimensional marginals),
where of course $\varphi^{(0)}$ has distribution function $F$. Now Corollary
3.1 of \cite{Z11} guarantees that this carries over to $(B_{l_{j_{i}}})$, so
that $\mu(B_{l_{j_{i}}})\Phi_{B_{l_{j_{i}}}}\overset{\mu}{\Longrightarrow}%
\Phi$ with the same limit. Projecting onto the first component gives $\mu
(\mu(B_{l_{j_{i}}})\varphi_{B_{l_{j_{i}}}}\leq t)\Longrightarrow F(t)$, which
contradicts the $\vartheta$-estimate above.
\end{proof}%

%TCIMACRO{\TeXButton{VSs}{\vspace{0.3cm}}}%
%BeginExpansion
\vspace{0.3cm}%
%EndExpansion

Next, we show that there are always initial measures $\nu$ violating the LLT.

\begin{proof}
[\textbf{Proof of Proposition \ref{P_BadInitialDistrForLLTHit}}]Choose any
$\delta>0$. By Theorem \ref{T_LLT_Hit}, assumption (\ref{Eq_yxcvvcxy}) implies
(\ref{Eq_LLT_HTS_Simple}). In particular, $\mu(\varphi_{A_{l}}=k_{l}%
)\rightarrow0,$ and we can pick $l_{j}\nearrow\infty$ s.t. $\sum_{j\geq1}%
\mu(\varphi_{A_{l_{j}}}=k_{l_{j}})<\delta$. This ensures that $B:=%
%TCIMACRO{\tbigcap \nolimits_{j\geq1}}%
%BeginExpansion
{\textstyle\bigcap\nolimits_{j\geq1}}
%EndExpansion
\{\varphi_{A_{l_{j}}}=k_{l_{j}}\}^{c}$ satisfies $\mu(B)>1-\delta$, and it is
immediate that $\mu_{B}(\varphi_{A_{l_{j}}}=k_{l_{j}})=0$ for $j\geq1$,
proving (\ref{Eq_yxcvvcxyLLTfails}). On the other hand, Proposition
\ref{P_StrongDistrCgeHittingTimes} guarantees that (\ref{Eq_yxcvvcxy}) implies
(\ref{Eq_dfgvdfvjhacvvvvvvvvvvvvvvvvv}).
\end{proof}%

%TCIMACRO{\TeXButton{VSs}{\vspace{0.3cm}}}%
%BeginExpansion
\vspace{0.3cm}%
%EndExpansion
Here is an observation required for the proof of the next theorems.

\begin{lemma}
\label{L_Lemma}Let $T$ be an ergodic measure-preserving map on the probability
space $(X,\mathcal{A},\mu)$, and $A\in\mathcal{A}$. Then
\begin{equation}
1_{\{\varphi_{A}=k\}}=1_{\{\varphi_{A}=k-j\}}\circ T^{j}-1_{\{\varphi_{A}\leq
j\}\cap\{\varphi_{A}\circ T^{j}=k-j\}}\text{ \quad for }0\leq j<k\text{,}%
\label{Eq_LLL1}%
\end{equation}
and
\begin{equation}
\mu(\{\varphi_{A}\leq j\}\cap\{\varphi_{A}\circ T^{j}=m\})=\mu(A\cap
\{m\leq\varphi_{A}<m+j\})\text{ \quad for }j,m\geq1\text{.}\label{Eq_LLL2}%
\end{equation}

\end{lemma}

\begin{proof}
For the first assertion we need only observe that
\[
\{\varphi_{A}=k\}=\{\varphi_{A}>j\}\cap\{\varphi_{A}\circ T^{j}=k-j\}\text{
\quad for }0\leq j<k\text{.}%
\]
Turning to the second statement, note that
\[
\{\varphi_{A}\leq j\}\cap\{\varphi_{A}\circ T^{j}=m\}=\bigcup_{i=1}^{j}%
T^{-i}(A\cap\{\varphi_{A}=m+j-i\})\text{ \quad(disjoint)}%
\]
(decompose the left-hand event according to the last time $i$ in
$\{1,\ldots,j\}$ at which the orbit visits $A$). Hence, by $T$-invariance of
$\mu$,
\[
\mu(\{\varphi_{A}\leq j\}\cap\{\varphi_{A}\circ T^{j}=m\})=\sum_{i=1}^{j}%
\mu(A\cap\{\varphi_{A}=m+j-i\})\text{,}%
\]
which proves the identity in (\ref{Eq_LLL2}).
\end{proof}%

%TCIMACRO{\TeXButton{VSs}{\vspace{0.3cm}}}%
%BeginExpansion
\vspace{0.3cm}%
%EndExpansion

We can then provide the

\begin{proof}
[\textbf{Proof of Proposition \ref{T_LLT_Hit_DensiFromU}}]\textbf{(i)} To
prove (a) observe that
\[
\left\vert \nu(\varphi_{A_{l}}=k_{l})-\mu(\varphi_{A_{l}}=k_{l})\right\vert
=o(\mu(A_{l}))\text{ \quad}%
\begin{array}
[c]{c}%
\text{as }l\rightarrow\infty\text{, uniformly in}\\
\nu\in\mathcal{V}\text{ and }(k_{l})\in\mathcal{K}_{\delta}(A_{l})\text{.}%
\end{array}
\text{ }%
\]
(Otherwise there exist $\eta>0$, $\nu^{(j)}\in\mathcal{V}$, and $(k_{l}%
^{(j)})_{l\geq1}\in\mathcal{K}_{\delta}(A_{l})$ such that
\[
\left\vert \nu^{(l)}(\varphi_{A_{l}}=k_{l}^{(l)})-\mu(\varphi_{A_{l}}%
=k_{l}^{(l)})\right\vert \geq\eta\mu(A_{l})\text{ \quad for infinitely many
}l\geq1\text{.}%
\]
Set $k_{l}:=k_{l}^{(l)}$ for $l\geq1$, then $(k_{l})\in\mathcal{K}_{\delta
}(A_{l})$ with $\left\vert \nu^{(l)}(\varphi_{A_{l}}=k_{l})-\mu(\varphi
_{A_{l}}=k_{l})\right\vert \geq\eta\mu(A_{l})$ for infinitely many $l\geq1$,
contradicting the definition of an LLT-uniform family $\mathcal{V}$.)
Assertion (a) of the present proposition is now immediate from statement (c)
of Theorem \ref{T_LLT_Hit}.\newline\newline\textbf{(ii)} Turning to part (b),
fix any sequence $(k_{l})\in\mathcal{K}(A_{l})$, w.l.o.g. with $k_{l}\geq2$.
Then the $t_{l}:=\mu(A_{l})k_{l}\ $satisfy $(t_{l})_{l\geq1}\subseteq
K\subseteq(0,\infty)$ for some compact set $K$. Choose any sequence
$(j_{l})_{l\geq1}$ in $\mathbb{N}$ such that $j_{l}<k_{l}$ for $l\geq l_{0}$
and $j_{l}\nearrow\infty$ while $j_{l}=o(k_{l})$ as $l\rightarrow\infty$. Set
$\kappa_{l}:=k_{l}-j_{l}$, then the sequence given by $s_{l}:=\mu(A_{l}%
)\kappa_{l}$, $l\geq1$, satisfies $(s_{l})_{l\geq l_{1}}\subseteq K^{\prime}$
for some $l_{1}\geq l_{0}$ and some compact set $K^{\prime}\subseteq
(0,\infty)$ containing $K$.\newline\newline\textbf{(iii)} The main point now
is to show that
\begin{equation}
\left\vert \int(1_{\{\varphi_{A_{l}}=k_{l}\}}-1_{\{\varphi_{A_{l}}=\kappa
_{l}\}}\circ T^{j_{l}})\,u\,d\mu\right\vert =o\left(  \mu(A_{l})\right)
\text{ }%
\begin{array}
[c]{c}%
\text{as }l\rightarrow\infty\text{,}\\
\text{uniformly in }u\in\mathcal{U}\text{.}%
\end{array}
\text{ }\label{Eq_skip_j_steps}%
\end{equation}
Since $T^{-1}\{\varphi_{A}=k-1\}=\{\varphi_{A}=k\}\cup T^{-1}(A\cap
\{\varphi_{A}=k-1\})$ (disjoint) for $k\geq2$, we get $1_{\{\varphi_{A}%
=k\}}=1_{\{\varphi_{A}=k-1\}}\circ T-1_{A\cap\{\varphi_{A}=k-1\}}\circ T$ for
such $k$. Therefore the left-hand side of (\ref{Eq_skip_j_steps}) is bounded
above by%
\begin{equation}
\left\vert \int(1_{\{\varphi_{A_{l}}=\kappa_{l}\}}\circ T^{j_{l}}%
-1_{\{\varphi_{A_{l}}=k_{l}-1\}}\circ T)\,u\,d\mu\right\vert +\int
(1_{A_{l}\cap\{\varphi_{A_{l}}=k_{l}-1\}}\circ T)\,u\,d\mu\text{.}%
\label{Eq_BBBound}%
\end{equation}
\newline In view of (\ref{Eq_LLL1}) from Lemma \ref{L_Lemma}, we have
\[
1_{\{\varphi_{A_{l}}=\kappa_{l}\}}\circ T^{j_{l}-1}=1_{\{\varphi_{A_{l}}%
=k_{l}-1\}}+1_{\{\varphi_{A_{l}}\leq j_{l}-1\}\cap\{\varphi_{A_{l}}\circ
T^{j_{l}-1}=\kappa_{l}\}}\text{,}%
\]
so that the first expression in (\ref{Eq_BBBound}) equals%
\begin{align*}
&  \left\vert \int(1_{\{\varphi_{A_{l}}=\kappa_{l}\}}\circ T^{j_{l}%
-1}-1_{\{\varphi_{A_{l}}=k_{l}-1\}})\,\widehat{T}u\,d\mu\right\vert \\
&  =\int1_{\{\varphi_{A_{l}}\leq j_{l}-1\}\cap\{\varphi_{A_{l}}\circ
T^{j_{l}-1}=\kappa_{l}\}}\,\widehat{T}u\,d\mu\\
&  \leq\left\Vert \widehat{T}u\right\Vert _{\infty}\,\mu\left(  \{\varphi
_{A_{l}}\leq j_{l}-1\}\cap\{\varphi_{A_{l}}\circ T^{j_{l}-1}=\kappa
_{l}\}\right) \\
&  \leq\sup\limits_{u\in\mathcal{U}}\left\Vert \widehat{T}u\right\Vert
_{\infty}\,\mu\left(  A_{l}\cap\{\kappa_{l}\leq\varphi_{A_{l}}\leq
k_{l}\}\right)  \text{,}%
\end{align*}
\newline where the last step uses (\ref{Eq_LLL2}) of Lemma \ref{L_Lemma}. To
see that
\[
\left\vert \int(1_{\{\varphi_{A_{l}}=\kappa_{l}\}}\circ T^{j_{l}}%
-1_{\{\varphi_{A_{l}}=k_{l}-1\}}\circ T)\,u\,d\mu\right\vert =o\left(
\mu(A_{l})\right)  \text{ }%
\begin{array}
[c]{c}%
\text{as }l\rightarrow\infty\text{,}\\
\text{uniformly in }u\in\mathcal{U}\text{,}%
\end{array}
\]
\newline it thus suffices to validate
\[
\mu_{A_{l}}(\kappa_{l}\leq\varphi_{A_{l}}\leq k_{l})=\mu_{A_{l}}(s_{l}\leq
\mu(A_{l})\varphi_{A_{l}}\leq t_{l})\longrightarrow0\text{ \quad as
}l\rightarrow\infty\text{.}%
\]
But recalling Theorem \ref{T_HittingTimeVsReturnTimeLimits} we know
$\mu_{A_{l}}(\mu(A_{l})\varphi_{A_{l}}\leq s)\Longrightarrow\widetilde{F}(s)$,
and since $s_{l}\sim t_{l}$ with all points in $K^{\prime}$, and
$\widetilde{F}$ uniformly continuous on $K^{\prime}$, this follows at once.

It remains to control the second expression in (\ref{Eq_BBBound}) and to
prove
\[
\int(1_{A_{l}\cap\{\varphi_{A_{l}}=k_{l}-1\}}\circ T)\,u\,d\mu=o\left(
\mu(A_{l})\right)  \text{ }%
\begin{array}
[c]{c}%
\text{as }l\rightarrow\infty\text{,}\\
\text{uniformly in }u\in\mathcal{U}\text{.}%
\end{array}
\]
But the left-hand side is%
\[
\int1_{A_{l}\cap\{\varphi_{A_{l}}=k_{l}-1\}}\,\widehat{T}u\,d\mu\leq
\sup\limits_{u\in\mathcal{U}}\left\Vert \widehat{T}u\right\Vert _{\infty}%
\,\mu_{A_{l}}(\varphi_{A_{l}}=k_{l}-1)\,\mu(A_{l})
\]
with $\,\mu_{A_{l}}(\varphi_{A_{l}}=k_{l}-1)\leq\mu_{A_{l}}(\kappa_{l}%
\leq\varphi_{A_{l}}\leq k_{l})\rightarrow0$, as required.\newline%
\newline\textbf{(iv)} Next, we check that
\begin{equation}
\left\vert \int(1_{\{\varphi_{A_{l}}=\kappa_{l}\}}\circ T^{j_{l}}%
)\,(u-1_{X})\,d\mu\right\vert =o\left(  \mu(A_{l})\right)  \text{ }%
\begin{array}
[c]{c}%
\text{as }l\rightarrow\infty\text{,}\\
\text{uniformly in }u\in\mathcal{U}\text{.}%
\end{array}
\label{Eq_mu_to_u_j_steps}%
\end{equation}
Here, the left-hand expression coincides with
\[
\left\vert \int1_{\{\varphi_{A_{l}}=\kappa_{l}\}}\,\widehat{T}^{j_{l}}%
(u-1_{X})\,d\mu\right\vert \leq\,\parallel\widehat{T}^{j_{l}}u-1_{X}%
\parallel_{\infty}\mu(\varphi_{A_{l}}=\kappa_{l})\text{,}%
\]
and due to assumption (\ref{Eq_UnifUnifCge}) and $j_{l}\nearrow\infty$ we need
only prove that
\begin{equation}
\mu(\varphi_{A_{l}}=\kappa_{l})=O\left(  \mu(A_{l})\right)  \text{ \quad as
}l\rightarrow\infty\text{.}\label{Eq_gdsvchasvchdsgv}%
\end{equation}
But under the present assumptions, Theorem \ref{T_LLT_Hit} c) applies, which
guarantees that $\mu(\varphi_{A_{l}}=\kappa_{l})\sim F^{\prime}(\mu
(A_{l})\kappa_{l})\cdot\mu(A_{l})$. This implies (\ref{Eq_gdsvchasvchdsgv})
since $F^{\prime}\leq1$, see Remark \ref{Rem_PropF}.\newline\newline%
\textbf{(v)} It remains to show that
\begin{equation}
\left\vert u\odot\mu(\varphi_{A_{l}}=k_{l})-\mu(\varphi_{A_{l}}=k_{l}%
)\right\vert =o\left(  \mu(A_{l})\right)  \text{ }%
\begin{array}
[c]{c}%
\text{as }l\rightarrow\infty\text{,}\\
\text{uniformly in }u\in\mathcal{U}\text{.}%
\end{array}
\label{Eq_mu_to_u}%
\end{equation}
The left-hand expression in (\ref{Eq_mu_to_u}) is bounded above by
\begin{multline*}
\left\vert u\odot\mu(\varphi_{A_{l}}=k_{l})-u\odot\mu(\varphi_{A_{l}}\circ
T^{j_{l}}=\kappa_{l})\right\vert \\
+\left\vert u\odot\mu(\varphi_{A_{l}}\circ T^{j_{l}}=\kappa_{l})-\mu
(\varphi_{A_{l}}\circ T^{j_{l}}=\kappa_{l})\right\vert \\
+\left\vert \mu(\varphi_{A_{l}}\circ T^{j_{l}}=\kappa_{l})-\mu(\varphi_{A_{l}%
}=k_{l})\right\vert \text{.}%
\end{multline*}
Now (\ref{Eq_mu_to_u_j_steps}) shows that the middle term is $o\left(
\mu(A_{l})\right)  $ uniformly in $u\in\mathcal{U}$ as $l\rightarrow\infty$,
while (\ref{Eq_skip_j_steps}) confirms that the same is true for the first
term. But since we can assume w.l.o.g. that $1_{X}\in\mathcal{U}$,
(\ref{Eq_skip_j_steps}) also applies to show that the last term is $o\left(
\mu(A_{l})\right)  $ uniformly in $u\in\mathcal{U}$. This proves
(\ref{Eq_mu_to_u}).
\end{proof}%

%TCIMACRO{\TeXButton{VSs}{\vspace{0.3cm}}}%
%BeginExpansion
\vspace{0.3cm}%
%EndExpansion

We can now turn to LLTs for return times, starting with the

\begin{proof}
[\textbf{Proof of Proposition \ref{P_ookmmmmms}}]\textbf{(i)} We first observe
that
\begin{equation}
\varphi_{B}\neq k\text{\quad on }B:=A\cap\{\varphi_{A}\neq k\}\text{ \qquad
for }A\in\mathcal{A}\text{ and }k\geq1\text{.}\label{Eq_yepyep}%
\end{equation}
Indeed, whenever $B\subseteq A$, then $\varphi_{B}\geq\varphi_{A}$ and
\begin{equation}
\varphi_{B}=\left\{
\begin{array}
[c]{ll}%
\varphi_{A} & \text{on }B\cap T_{A}^{-1}B\text{,}\\
\varphi_{A}+\varphi_{B}\circ T_{A} & \text{on }B\cap T_{A}^{-1}(A\setminus
B)\text{.}%
\end{array}
\right.
\end{equation}
In the special case of $B:=A\cap\{\varphi_{A}\neq k\}$ this immediately gives
\[
\varphi_{B}=\varphi_{A}\neq k\quad\text{on }B\cap T_{A}^{-1}B\text{.}%
\]
On the other hand, for this choice of $B$, we have $\varphi_{A}\circ T_{A}=k $
on $B\cap T_{A}^{-1}(A\setminus B)\subseteq T_{A}^{-1}\{\varphi_{A}=k\}$, and
hence
\[
\varphi_{B}=\varphi_{A}+\varphi_{B}\circ T_{A}\geq1+\varphi_{A}\circ
T_{A}=1+k\quad\text{on }B\cap T_{A}^{-1}(A\setminus B)\text{.}%
\]
\textbf{(ii)} Now, given $(A_{l})$ and $(k_{l})$ as in the statement of the
proposition, take $B_{l}:=A_{l}\cap\{\varphi_{A_{l}}\neq k_{l}\}$. According
to (\ref{Eq_yepyep}), these satisfy $\mu_{B_{l}}(\varphi_{B_{l}}=k_{l})=0$ for
$l\geq1$.

Since $\mu(A_{l})k_{l}\ $converges to a continuity point of $\widetilde{F}$,
(\ref{Eq_kjhgfdsalkjhgfds}) entails $\mu_{A_{l}}(\mu(A_{l})\varphi_{A_{l}}%
=\mu(A_{l})k_{l})\rightarrow0$ as $l\rightarrow\infty$. But $\mu_{A_{l}%
}(\varphi_{A_{l}}=k_{l})=\mu_{A_{l}}(A_{l}\setminus B_{l})$, proving
$\mu(B_{l})\sim\mu(A_{l})$. The latter allows us to deduce
(\ref{Eq_zhbgzhbgzhbgzhbg}) from (\ref{Eq_kjhgfdsalkjhgfds}), see Remark 2.3
in \cite{Z11}.
\end{proof}%

%TCIMACRO{\TeXButton{VSs}{\vspace{0.3cm}}}%
%BeginExpansion
\vspace{0.3cm}%
%EndExpansion

Next, we prove Theorem \ref{T_LLT_RetGeneral} (which contains Theorem
\ref{T_LLT_Ret0}).

\begin{proof}
[\textbf{Proof of Theorem \ref{T_LLT_RetGeneral}}]\textbf{(a)} Take any
$(k_{l})\in\mathcal{K}(A_{l})$. According to Theorem \ref{T_AutoLLT_Hit}, the
LLT for hitting times holds in the form of (\ref{Eq_AutoLLT_HitExp}).
Moreover, since $\inf_{l\geq1}\mu(A_{l})k_{l}>0$, assumption
(\ref{Eq_ReturningPartGeneral2}) implies that $\sup_{\nu_{l}\in\mathcal{V}%
_{l}}\nu_{l}^{\bullet}\left(  \varphi_{A_{l}}=k_{l}\right)  =0$ for $l\geq
l_{0}$. Therefore (\ref{Eq_ReturningPartGeneral1}) plus
(\ref{Eq_RetEquivHitGeneral}) imply that%
\begin{align}
\nu_{l}(\varphi_{A_{l}}=k_{l}) &  =\nu_{l}(A_{l}^{\bullet})\nu_{l}^{\bullet
}(\varphi_{A_{l}}=k_{l})+\nu_{l}(A_{l}^{\circ})\nu_{l}^{\circ}(\varphi_{A_{l}%
}=k_{l})\nonumber\\
&  \sim\theta\,\nu_{l}^{\circ}(\varphi_{A_{l}}=k_{l})\sim\theta\,\mu
(\varphi_{A_{l}}=k_{l})\nonumber\\
&  \sim\theta^{2}e^{-\theta\mu(A_{l})k_{l}}\,\mu(A_{l})\text{ \quad uniformly
in }(\nu_{l})\in%
%TCIMACRO{\tprod \nolimits_{l\geq1}}%
%BeginExpansion
{\textstyle\prod\nolimits_{l\geq1}}
%EndExpansion
\mathcal{V}_{l}\text{.}\label{Eq_Intermediate}%
\end{align}
To see that this asymptotic relation actually holds uniformly in $(\nu_{l})\in%
%TCIMACRO{\tprod \nolimits_{l\geq1}}%
%BeginExpansion
{\textstyle\prod\nolimits_{l\geq1}}
%EndExpansion
\mathcal{V}_{l}$ \emph{and} $(k_{l})\in\mathcal{K}_{\delta}(A_{l})$, assume
the contrary. This means that there are some $\varepsilon>0$, infinitely many
indices $1\leq l_{1}<l_{2}<\ldots$ as well as sequences $(\nu_{l}^{(i)})\in%
%TCIMACRO{\tprod \nolimits_{l\geq1}}%
%BeginExpansion
{\textstyle\prod\nolimits_{l\geq1}}
%EndExpansion
\mathcal{V}_{l}$ and $(k_{l}^{(i)})\in\mathcal{K}_{\delta}(A_{l})$, $i\geq1$,
such that
\begin{equation}
\left\vert \frac{\nu_{l_{i}}^{(i)}(\varphi_{A_{l_{i}}}=k_{l_{i}}^{(i)}%
)}{\theta^{2}e^{-\theta\mu(A_{l_{i}})k_{l_{i}}^{(i)}}\,\mu(A_{l_{i}}%
)}-1\right\vert \geq\varepsilon\text{ \quad for }i\geq1\text{.}%
\label{Eq_Contrary}%
\end{equation}
Define $A_{i}^{\prime}:=A_{l_{i}}$, $\mathcal{V}_{l}^{\prime}:=\mathcal{V}%
_{l_{i}}$, $(\nu_{i}^{\prime})\in%
%TCIMACRO{\tprod \nolimits_{i\geq1}}%
%BeginExpansion
{\textstyle\prod\nolimits_{i\geq1}}
%EndExpansion
\mathcal{V}_{i}^{\prime}$ with $\nu_{i}^{\prime}:=\nu_{l_{i}}^{(i)}$, and
$(k_{i}^{\prime})\in\mathcal{K}_{\delta}(A_{i}^{\prime})$ with $k_{i}^{\prime
}:=k_{l_{i}}^{(i)}$. Then, by the conclusion (\ref{Eq_Intermediate}) of the
previous paragraph,
\[
\nu_{i}^{\prime}(\varphi_{A_{i}^{\prime}}=k_{i}^{\prime})\sim\theta
^{2}e^{-\theta\mu(A_{i}^{\prime})k_{i}^{\prime}}\,\mu(A_{i}^{\prime})\text{
\quad as }i\rightarrow\infty\text{,}%
\]
which clearly contradicts (\ref{Eq_Contrary}).\newline\newline\textbf{(b)} Fix
a sequence $(k_{l})\in\mathcal{K}_{\delta}(A_{l})$ for some $\delta$. Observe
first that there exist $\varepsilon_{l}\in(0,1]$ such that $\varepsilon
_{l}\rightarrow0$ and%
\begin{gather}
\nu_{l}^{\circ}(\,\tau_{l}>\varepsilon_{l}k_{l})\leq\varepsilon_{l}\text{
\quad and \quad}\nu_{l}^{\circ}(\,\{\tau_{l}>\varepsilon_{l}k_{l}%
\}\cap\{\varphi_{A_{l}}=k_{l}\})\leq\varepsilon_{l}\mu(A_{l})\nonumber\\
\text{for }l\geq1\text{ and }(\nu_{l})\in%
%TCIMACRO{\tprod \nolimits_{l\geq1}}%
%BeginExpansion
{\textstyle\prod\nolimits_{l\geq1}}
%EndExpansion
\mathcal{V}_{l}\text{.}\label{Eq_OrderOfTheTauLGeneral3}%
\end{gather}
Indeed, as a consequence of (\ref{Eq_OrderOfTheTauLGeneral1}) (or the more
generous condition (\ref{Eq_OrderOfTheTauLGeneral5})) there are indices
$1=l_{0}<l_{1}<\ldots$ such that for every $j\geq0$,
\begin{gather}
\nu_{l}^{\circ}\left(  \tau_{l}>\frac{k_{l}}{2^{j}}\right)  \leq\frac{1}%
{2^{j}}\text{ \quad and \quad}\nu_{l}^{\circ}\left(  \left\{  \tau_{l}%
>\frac{k_{l}}{2^{j}}\right\}  \cap\{\varphi_{A_{l}}=k_{l}\}\right)  \leq
\frac{\mu(A_{l})}{2^{j}}\nonumber\\
\text{for }l\geq l_{j}\text{ and }(\nu_{l})\in%
%TCIMACRO{\tprod \nolimits_{l\geq1}}%
%BeginExpansion
{\textstyle\prod\nolimits_{l\geq1}}
%EndExpansion
\mathcal{V}_{l}\text{.}\label{Eq_OrderOfTheTauLGeneral2}%
\end{gather}
Having found these, we can take $\varepsilon_{l}:=2^{-j}$ for $l\in
\{l_{j},\ldots,l_{j+1}-1\}$. Due to (\ref{Eq_OrderOfTheTauLGeneral3}) we have%
\begin{align}
& \left\vert \nu_{l}^{\circ}(\varphi_{A_{l}}=k_{l})-\nu_{l}^{\circ}(\{\tau
_{l}\leq\varepsilon_{l}k_{l}\}\cap\{\varphi_{A_{l}}=k_{l}\})\right\vert
\label{Eq_AAAA}\\
& =\nu_{l}^{\circ}(\{\tau_{l}>\varepsilon_{l}k_{l}\}\cap\{\varphi_{A_{l}%
}=k_{l}\})\nonumber\\
& =o(\mu(A_{l}))\text{ \ uniformly in }(\nu_{l})\in%
%TCIMACRO{\tprod \nolimits_{l\geq1}}%
%BeginExpansion
{\textstyle\prod\nolimits_{l\geq1}}
%EndExpansion
\mathcal{V}_{l}\text{.}\nonumber
\end{align}
Recall (\ref{Eq_LLL1}) to see that for $j<k_{l}$,
\begin{align*}
\nu_{l}^{\circ}(\{\tau_{l} &  =j\}\cap T^{-j}\{\varphi_{A_{l}}=k_{l}%
-j\})-\nu_{l}^{\circ}(\{\tau_{l}=j\}\cap\{\varphi_{A_{l}}=k_{l}\})\\
&  =\nu_{l}^{\circ}(\{\tau_{l}=j\}\cap\{\varphi_{A_{l}}\leq j\}\cap
T^{-j}\{\varphi_{A_{l}}=k_{l}-j\})\\
&  \leq\nu_{l}^{\circ}(\{\tau_{l}=j\}\cap\{\varphi_{A_{l}}\leq j\}\cap
T^{-k_{l}}A_{l})\text{.}%
\end{align*}
Consequently, summing over $j\in\{0,\ldots,\left\lfloor \varepsilon_{l}%
k_{l}\right\rfloor \}$, we find via (\ref{Eq_NothingHappensTooQuicklyGeneral})
that%
\begin{align}
&  \left\vert \nu_{l}^{\circ}(\{\tau_{l}\leq\varepsilon_{l}k_{l}%
\}\cap\{\varphi_{A_{l}}=k_{l}\})-\nu_{l}^{\circ}(\{\tau_{l}\leq\varepsilon
_{l}k_{l}\}\cap\{\varphi_{A_{l}}\circ T^{\tau_{l}}=k_{l}-\tau_{l}%
\})\right\vert \label{Eq_AAAAA}\\
&  \leq\nu_{l}^{\circ}(\{\tau_{l}\leq\varepsilon_{l}k_{l}\}\cap\{\varphi
_{A_{l}}\leq\tau_{l}\}\cap T^{-k_{l}}A_{l})\nonumber\\
&  =o(\mu(A_{l}))\text{ \ uniformly in }(\nu_{l})\in%
%TCIMACRO{\tprod \nolimits_{l\geq1}}%
%BeginExpansion
{\textstyle\prod\nolimits_{l\geq1}}
%EndExpansion
\mathcal{V}_{l}\text{.}\nonumber
\end{align}
Next,%
\[
\nu_{l}^{\circ}(\{\tau_{l}\leq\varepsilon_{l}k_{l}\}\cap\{\varphi_{A_{l}}\circ
T^{\tau_{l}}=k_{l}-\tau_{l}\})=\sum_{j=0}^{\left\lfloor \varepsilon_{l}%
k_{l}\right\rfloor }\nu_{l}^{\circ}(\{\tau_{l}=j\})\cdot(\nu_{l,\{\tau
_{l}=j\}}^{\circ}\circ T^{-j})(\varphi_{A_{l}}=k_{l}-j)\text{,}%
\]
and in view of (\ref{Eq_SailingAfterTauGeneral}), property (\ref{Eq_LLTgood})
guarantees
\[
\nu_{l,\{\tau_{l}=j_{l}\}}^{\circ}\circ T^{-j_{l}}(\varphi_{A_{l}}=k_{l}%
-j_{l})\sim\mu(\varphi_{A_{l}}=k_{l}-j_{l})\text{ \quad}%
\begin{array}
[c]{c}%
\text{as }l\rightarrow\infty\text{, uniformly in }(\nu_{l})\in%
%TCIMACRO{\tprod \nolimits_{l\geq1}}%
%BeginExpansion
{\textstyle\prod\nolimits_{l\geq1}}
%EndExpansion
\mathcal{V}_{l}\\
\text{and }(j_{l})\in%
%TCIMACRO{\tprod \nolimits_{l\geq1}}%
%BeginExpansion
{\textstyle\prod\nolimits_{l\geq1}}
%EndExpansion
\{0,\ldots,\left\lfloor \varepsilon_{l}k_{l}\right\rfloor \}
\end{array}
\]
(since for any such sequence $(j_{l})$ we have $(k_{l}-j_{l})\in
\mathcal{K}_{(1-\epsilon)\delta}(A_{l})$.) Moreover,
\[
\mu(\varphi_{A_{l}}=k_{l}-j_{l})\sim\mu(\varphi_{A_{l}}=k_{l})\text{ \quad}%
\begin{array}
[c]{c}%
\text{as }l\rightarrow\infty\text{, uniformly in}\\
(j_{l})\in%
%TCIMACRO{\tprod \nolimits_{l\geq1}}%
%BeginExpansion
{\textstyle\prod\nolimits_{l\geq1}}
%EndExpansion
\{0,\ldots,\left\lfloor \varepsilon_{l}k_{l}\right\rfloor \}\text{,}%
\end{array}
\]
because of (\ref{Eq_AutoLLT_HitExp}). Together these show that
\begin{align}
\nu_{l}^{\circ}(\{\tau_{l}  & \leq\varepsilon_{l}k_{l}\}\cap\{\varphi_{A_{l}%
}\circ T^{\tau_{l}}=k_{l}-\tau_{l}\})\label{Eq_AAAAAA}\\
& \sim\nu_{l}^{\circ}(\{\tau_{l}\leq\varepsilon_{l}k_{l}\})\,\mu
(\varphi_{A_{l}}=k_{l})\nonumber\\
\text{as }l  & \rightarrow\infty\text{, uniformly in }(\nu_{l})\in%
%TCIMACRO{\tprod \nolimits_{l\geq1}}%
%BeginExpansion
{\textstyle\prod\nolimits_{l\geq1}}
%EndExpansion
\mathcal{V}_{l}\text{.}\nonumber
\end{align}
Finally, (\ref{Eq_OrderOfTheTauLGeneral3}) also ensures that
\begin{equation}
\nu_{l}^{\circ}(\{\tau_{l}\leq\varepsilon_{l}k_{l}\})\longrightarrow1\text{ as
}l\rightarrow\infty\text{, uniformly in }(\nu_{l})\in%
%TCIMACRO{\tprod \nolimits_{l\geq1}}%
%BeginExpansion
{\textstyle\prod\nolimits_{l\geq1}}
%EndExpansion
\mathcal{V}_{l}\text{,}\label{Eq_AAAAAAA}%
\end{equation}
and combining (\ref{Eq_AAAA})-(\ref{Eq_AAAAAAA}) gives
(\ref{Eq_RetEquivHitGeneral}).
\end{proof}%

%TCIMACRO{\TeXButton{VSs}{\vspace{0.3cm}}}%
%BeginExpansion
\vspace{0.3cm}%
%EndExpansion

We can the continue with the

\begin{proof}
[\textbf{Proof of Theorem \ref{T_LLT_ConsecutiveTimes}}]The result about
consecutive times now follows by straightforward induction, and including the
positions does not pose serious difficulties either. \newline\newline%
\textbf{(i) statements a)} \textbf{and} \textbf{a')} Theorem
\ref{T_LLT_RetGeneral} immediately gives ($\clubsuit_{1}$). Under the extra
assumption (\ref{Eq_InvarianceVls}), with $(k_{l})\in\mathcal{K}(A_{l})$, we
have $w_{l}:=\nu_{l,\{\varphi_{A_{l}}=k_{l}^{(1)}\}}\circ T_{A_{l}}^{-1}%
\in\mathcal{V}_{l}$, and hence, whenever $(C_{l}^{(1)})$ and $(c_{l}^{(1)})$
satisfy (\ref{Eq_AsyNuCl}),%
\begin{align*}
\nu_{l}\left(  \{\varphi_{A_{l}}=k_{l}^{(1)}\}\cap T_{A_{l}}^{-1}C_{l}%
^{(1)}\right)   & =\nu_{l}(\varphi_{A_{l}}=k_{l}^{(1)})\,w_{l}(C_{l}^{(1)})\\
& \sim\theta^{2}e^{-\theta\mu(A_{l})k_{l}^{(1)}}\,\mu(A_{l})\,c_{l}%
^{(1)}\text{,}%
\end{align*}
which holds uniformly in $(\nu_{l})\in%
%TCIMACRO{\tprod \nolimits_{l\geq1}}%
%BeginExpansion
{\textstyle\prod\nolimits_{l\geq1}}
%EndExpansion
\mathcal{V}_{l}$ (by \ref{Eq_AsyNuCl}) and $(k_{l})\in\mathcal{K}_{\delta
}(A_{l})$ (due to ($\clubsuit_{1}$)). This establishes ($\blacklozenge_{1}$).

For the inductive step fix some $d\geq1$ and assume ($\clubsuit_{d}$). For
$(k_{l}^{(1)}),\ldots,(k_{l}^{(d+1)})\in\mathcal{K}(A_{l})$ and $(\nu_{l})\in%
%TCIMACRO{\tprod \nolimits_{l\geq1}}%
%BeginExpansion
{\textstyle\prod\nolimits_{l\geq1}}
%EndExpansion
\mathcal{V}_{l}$,%
\[
\bigcap_{j=1}^{d+1}T_{A_{l}}^{-(j-1)}\{\varphi_{A_{l}}=k_{l}^{(j)}%
\}=\{\varphi_{A_{l}}=k_{l}^{(1)}\}\cap T_{A_{l}}^{-1}\left(  \bigcap_{j=1}%
^{d}T_{A_{l}}^{-(j-1)}\{\varphi_{A_{l}}=k_{l}^{(j+1)}\}\right)  \text{.}%
\]
Writing $\widetilde{\nu}_{l}:=\nu_{l,\{\varphi_{A_{l}}=k_{l}^{(1)}\}}\circ
T_{A_{l}}^{-1}$, which by (\ref{Eq_InvarianceVls}) belongs to $\mathcal{V}%
_{l}$, we find that
\begin{align*}
\nu_{l}\left(  \bigcap_{j=1}^{d+1}T_{A_{l}}^{-(j-1)}\{\varphi_{A_{l}}%
=k_{l}^{(j)}\}\right)   & =\nu_{l}(\varphi_{A_{l}}=k_{l}^{(1)})\,\widetilde
{\nu}_{l}\left(  \bigcap_{j=1}^{d}T_{A_{l}}^{-(j-1)}\{\varphi_{A_{l}}%
=k_{l}^{(j+1)}\}\right) \\
& \sim\theta^{2}e^{-\theta\mu(A_{l})k_{l}^{(1)}}\,\mu(A_{l})\,\theta
^{2d}e^{-\theta\mu(A_{l})\sum_{j=1}^{d}k_{l}^{(j+1)}}\,\mu(A_{l})^{d}%
\end{align*}
as $l\rightarrow\infty$, where ($\clubsuit_{1}$) and ($\clubsuit_{d}$)
guarantee that the asymptotics is uniform in $(\nu_{l})\in%
%TCIMACRO{\tprod \nolimits_{l\geq1}}%
%BeginExpansion
{\textstyle\prod\nolimits_{l\geq1}}
%EndExpansion
\mathcal{V}_{l}$ and $(k_{l}^{(1)}),\ldots,(k_{l}^{(d+1)})\in\mathcal{K}%
_{\delta}(A_{l})$. This proves ($\clubsuit_{d+1}$). In the situation of a'),
assume ($\blacklozenge_{d}$). Then, for $(C_{l}^{(j)})$ and $(c_{l}^{(j)})$
satisfying (\ref{Eq_kysdjhfbhjeabgfhjbevb}) and (\ref{Eq_AsyNuCl}), and,
$w_{l}$ as defined above, we get $w_{l,C_{l}^{(1)}}\in\mathcal{V}_{l}$, so
that%
\begin{align*}
& \nu_{l}\left(  \bigcap_{j=1}^{d+1}T_{A_{l}}^{-(j-1)}\left(  \{\varphi
_{A_{l}}=k_{l}^{(j)}\}\cap T_{A_{l}}^{-1}C_{l}^{(j)}\right)  \right) \\
& =\nu_{l}\left(  \{\varphi_{A_{l}}=k_{l}^{(1)}\}\cap T_{A_{l}}^{-1}%
C_{l}^{(1)}\right)  \,w_{l,C_{l}^{(1)}}\left(  \bigcap_{j=1}^{d}T_{A_{l}%
}^{-(j-1)}\left(  \{\varphi_{A_{l}}=k_{l}^{(j+1)}\}\cap T_{A_{l}}^{-1}%
C_{l}^{(j+1)}\right)  \right) \\
& \sim\theta^{2}e^{-\theta\mu(A_{l})k_{l}^{(1)}}\,\mu(A_{l})\,c_{l}%
^{(1)}\,\theta^{2d}e^{-\theta\mu(A_{l})\sum_{j=1}^{d}k_{l}^{(j+1)}}\,\mu
(A_{l})^{d}\prod_{j=1}^{d}c_{l}^{(j+1)}\text{,}%
\end{align*}
with asymptotics uniform in the required sense due to ($\blacklozenge_{1}$)
and ($\blacklozenge_{d}$).\newline\newline\textbf{(ii) statements b)}
\textbf{and} \textbf{b')} In view of Theorem \ref{T_AutoLLT_Hit} and
Proposition \ref{T_LLT_Hit_DensiFromU}, statement ($\spadesuit_{1}$) follows
at once from (\ref{Eq_AssmExpoHTS}). To prove ($\spadesuit_{d+1}$) for
$d\geq1$, we can argue as above: Letting $\widehat{\nu}_{l}:=\nu
_{\{\varphi_{A_{l}}=k_{l}^{(1)}\}}\circ T_{A_{l}}^{-1}\in\mathcal{V}_{l}$ we
find that $\widehat{\nu}_{l}\in\mathcal{V}_{l}$, and
\begin{align*}
\nu\left(  \bigcap_{j=1}^{d+1}T_{A_{l}}^{-(j-1)}\{\varphi_{A_{l}}=k_{l}%
^{(j)}\}\right)   & =\nu(\varphi_{A_{l}}=k_{l}^{(1)})\,\widehat{\nu}%
_{l}\left(  \bigcap_{j=1}^{d}T_{A_{l}}^{-(j-1)}\{\varphi_{A_{l}}=k_{l}%
^{(j+1)}\}\right) \\
& \sim\theta e^{-\theta\mu(A_{l})k_{l}^{(1)}}\,\mu(A_{l})\,\theta
^{2d}e^{-\theta\mu(A_{l})\sum_{j=1}^{d}k_{l}^{(j+1)}}\,\mu(A_{l})^{d}\text{,}%
\end{align*}
which holds uniformly in $\nu\in\mathcal{V}$ and $(k_{l}^{(1)}),\ldots
,(k_{l}^{(d+1)})\in\mathcal{K}_{\delta}(A_{l})$ by ($\spadesuit_{1}$) and
($\clubsuit_{d}$). In the situation of b'), ($\lozenge_{1}$) holds since%
\begin{align*}
\nu\left(  \{\varphi_{A_{l}}=k_{l}^{(1)}\}\cap T_{A_{l}}^{-1}C_{l}%
^{(1)}\right)   & =\nu(\varphi_{A_{l}}=k_{l}^{(1)})\,\widehat{\nu}_{l}%
(C_{l}^{(1)})\\
& \sim\theta e^{-\theta\mu(A_{l})k_{l}^{(1)}}\,\mu(A_{l})\,c_{l}^{(1)}%
\end{align*}
uniformly in $\nu\in\mathcal{V}$ and $(k_{l}^{(1)})\in\mathcal{K}_{\delta
}(A_{l})$, see (\ref{Eq_LLT_HTS_cpt_UVV}) and (\ref{Eq_Hourglass1}). Finally,
use the fact that $\widehat{\nu}_{l,C_{l}^{(j)}}\in\mathcal{V}_{l} $ by
(\ref{Eq_kysdjhfbhjeabgfhjbevb}) to see that%
\begin{align*}
& \nu\left(  \bigcap_{j=1}^{d+1}T_{A_{l}}^{-(j-1)}\left(  \{\varphi_{A_{l}%
}=k_{l}^{(j)}\}\cap T_{A_{l}}^{-1}C_{l}^{(j)}\right)  \right) \\
& =\nu\left(  \{\varphi_{A_{l}}=k_{l}^{(1)}\}\cap T_{A_{l}}^{-1}C_{l}%
^{(1)}\right)  \,\widehat{\nu}_{l,C_{l}^{(1)}}\left(  \bigcap_{j=1}%
^{d}T_{A_{l}}^{-(j-1)}\left(  \{\varphi_{A_{l}}=k_{l}^{(j+1)}\}\cap T_{A_{l}%
}^{-1}C_{l}^{(j+1)}\right)  \right) \\
& \sim\theta e^{-\theta\mu(A_{l})k_{l}^{(1)}}\,\mu(A_{l})\,c_{l}^{(1)}%
\,\theta^{2d}e^{-\theta\mu(A_{l})\sum_{j=1}^{d}k_{l}^{(j+1)}}\,\mu(A_{l}%
)^{d}\prod_{j=1}^{d}c_{l}^{(j+1)}\text{,}%
\end{align*}
with asymptotics uniform in the required sense due to ($\spadesuit_{1}$),
($\blacklozenge_{d}$), ($\lozenge_{1}$) and ($\lozenge_{d}$).
\end{proof}%

%TCIMACRO{\TeXButton{VSs}{\vspace{0.3cm}}}%
%BeginExpansion
\vspace{0.3cm}%
%EndExpansion

\section{Proofs for Section \ref{Sec_ConcreteLLTs}}%

%TCIMACRO{\TeXButton{NI}{\noindent}}%
%BeginExpansion
\noindent
%EndExpansion
\textbf{Gibbs-Markov maps.} Let $\mathfrak{S}=(X,\mathcal{A},\mu,T,\xi)$ be a
mixing probability preserving Gibbs-Markov system. We review a few well known
basic properties of such systems, all of which are obtained by elementary
routine arguments. Recalling the definition of $\mathrm{R}_{B}(u)$, we see
that for $A\subseteq B$ and $u\geq0$ with $\mathrm{R}_{B}(u)\leq r$ we have
$(1+r\,\mathrm{diam}(A))^{-1}u(y)\leq u(x)\leq(1+r\,\mathrm{diam}(A))\,u(y)$
when $x,y\in A$. If, in addition, $u$ does not vanish anywhere on $A$, then
\begin{equation}
\frac{1}{1+r\,\mathrm{diam}(A)}\leq\frac{u(x)}{u(y)}\leq1+r\,\mathrm{diam}%
(A)\text{ \quad}%
\begin{array}
[c]{c}%
\text{for }u\in\mathcal{U}(r)\text{ and}\\
x,y\in A\subseteq B\in\beta\text{,}%
\end{array}
\label{Eq_RRRegularity}%
\end{equation}
so that%
\begin{equation}
\frac{1}{1+r\,\mathrm{diam}(A)}\,\frac{\mu(A^{\prime})}{\mu(A)}\leq\frac
{\int_{A^{\prime}}u\,d\mu}{\int_{A}u\,d\mu}\leq(1+r\,\mathrm{diam}%
(A))\,\frac{\mu(A^{\prime})}{\mu(A)}\text{ }%
\begin{array}
[c]{c}%
\text{for }A,A^{\prime}\in\mathcal{A}\text{,}\\
A,A^{\prime}\subseteq B\text{.}%
\end{array}
\label{Eq_RatiosOnShortSets}%
\end{equation}
Note that (\ref{Eq_RegularityForGMSys}) says that the normalized image
measures $\mu_{Z}\circ v_{Z}$ with $n\geq1$ and $Z\in\xi_{n}$ have densities
$\widehat{T}^{n}\left(  \mu(Z)^{-1}1_{Z}\right)  $ belonging to $\mathcal{U}%
(r(\mathfrak{S}))$. More generally, one can find $r_{\ast}=r_{\ast
}(\mathfrak{S})\geq r(\mathfrak{S})$ such that (for suitable versions of the
relevant functions)%
\begin{equation}
\mathrm{R}_{T^{n}W}\left(  \widehat{T}^{n}\left(  1_{W}\,u\right)  \right)
\leq r\text{ \quad}%
\begin{array}
[c]{c}%
\text{for }W\subseteq Z\in\xi_{n}\text{ with }n\geq1\text{, }r\geq r_{\ast
}\text{,}\\
\text{and }u\geq0\text{ measurable with }\mathrm{R}_{Z}(u)\leq r\text{.}%
\end{array}
\label{Eq_NeoNeo}%
\end{equation}
In particular (since trivially $\mathrm{R}_{B}(cu)=\mathrm{R}_{B}(u)$ for any
constant $c>0$),
\begin{equation}
\widehat{T}^{n}\left(  \frac{1_{Z}\,u}{\int_{Z}u\,d\mu}\right)  \in
\mathcal{U}(r)\text{ \quad for }r\geq r_{\ast}\text{, }u\in\mathcal{U}%
(r)\text{, }n\geq1\text{ and }Z\in\xi_{n}\text{.}\label{Eq_ImageDensiCylGM}%
\end{equation}
Note that for any $B\subseteq X$ and $r\geq0$,
\begin{equation}
\{u:\mathrm{R}_{B}(u)\leq r\}\text{ is closed under countable convex
combinations. }\label{Eq_NeoNeoNeo}%
\end{equation}
It is clear that $\mathcal{U}(r)\subseteq\mathcal{U}(\widetilde{r})$ for
$0<r<\widetilde{r}$ and that also each
\begin{equation}
\mathcal{U}(r)\text{ is closed under countable convex combinations.}%
\label{Eq_ClosedCtbleConvex}%
\end{equation}
In addition to the facts reviewed above, there are constants $c_{\ast}(r)>0$
such that%
\begin{equation}
\left\Vert \widehat{T}^{n}\left(  \frac{1_{Z}\,u}{\int_{Z}u\,d\mu}\right)
\right\Vert _{\infty}\leq c_{\ast}(r)<\infty\text{\quad for }r>0\text{, }%
u\in\mathcal{U}(r)\text{, }n\geq1\text{ and }Z\in\xi_{n}\text{.}%
\label{Eq_puuuuh}%
\end{equation}
Observe that due to (\ref{Eq_ClosedCtbleConvex}) this estimate not only holds
for all $Z\in\xi_{n}$, but for all $Z=\bigcup_{Z^{\prime}\in\xi^{\prime}%
}Z^{\prime}$ where $\xi^{\prime}\subseteq\xi_{n}$. When applying the abstract
results of Section 2 to a GM-map we shall take, with $r\geq r^{\ast}$,%
\begin{equation}
\mathcal{U}:=\mathcal{U}(r)\text{ and }\mathcal{V}:=\{u\odot\mu:u\in
\mathcal{U}\}\text{.}%
\end{equation}
With this choice we ensure (see, for example, Proposition 1.4 of \cite{AD})
that
\begin{equation}
\widehat{T}\mathcal{U}=\widehat{T}\mathcal{U}(r)\text{ is bounded in
}L_{\infty}(\mu)\text{.}%
\end{equation}
Property (\ref{Eq_RegularityForGMSys}) also implies \emph{bounded distortion}
in that
\begin{equation}
\mu_{Z}(Z\cap T^{-n}A)=e^{\pm r}\mu_{T^{n}Z}(A)\text{ \quad for all }%
n\geq1\text{, }Z\in\xi_{n}\text{, and }A\in\mathcal{A}%
\label{Eq_AbstractAdlerInfty}%
\end{equation}
(with $a=c^{\pm1}b$ shorthand for $c^{-1}b\leq a\leq cb$). In particular,
\begin{equation}
\mu(Z\cap T^{-n}A)=\left(  \frac{e^{r}}{\flat}\right)  ^{\pm1}\mu(Z)\mu
(T^{n}Z\cap A)\text{ \quad for }n\geq1\text{, }Z\in\xi_{n}\text{, }%
A\in\mathcal{A}\text{.}\label{Eq_UpperDistortBound}%
\end{equation}
Also, a simple argument provides constants $\kappa\geq1$ and $q\in(0,1)$ such
that%
\begin{equation}
\mu(Z)\leq\kappa q^{n}\text{ \qquad for all }n\geq1\text{ and }Z\in\xi
_{n}\text{.}\label{Eq_GMMapsContractionOfCylMeasure}%
\end{equation}
It is well known that, for a mixing GM-map, densities $u$ from $\mathcal{U}%
(r)$ are attracted exponentially fast to the invariant density $1_{X}$ under
the action of the transfer operator. Specifically, $\kappa(r)$ and $q$ can be
chosen in such a way that
\begin{equation}
\parallel\widehat{T}^{n}u-1_{X}\parallel_{\infty}\leq\kappa(r)\,q^{n}\text{
\quad for }n\geq1\text{, }r>0\text{ and }u\in\mathcal{U}(r)\text{.}%
\label{Eq_ExpoCgeGM}%
\end{equation}
This follows from (even stronger results derived via) a spectral analysis of
$\widehat{T}$ on Lipschitz spaces (see e.g. \cite{AD}), but it can also be
derived in more elementary ways (similar to \cite{ZKuzmin}, say). Indeed, we
emphasize that the arguments of the present section only rely on the weaker
statement that
\begin{equation}
\text{for }r>0\text{,\quad}\parallel\widehat{T}^{n}u-1_{X}\parallel_{\infty
}\longrightarrow0\text{ \quad as }n\rightarrow\infty\text{, uniformly in }%
u\in\mathcal{U}(r)\text{. }%
\end{equation}
Via Proposition \ref{T_LLT_Hit_DensiFromU} the latter immediately shows that,
whatever $(A_{l})$,%
\begin{equation}
\text{for }r>0\text{,\quad}\mathcal{U}=\mathcal{U}(r)\text{ is LLT-uniform for
}(A_{l})\text{.}\label{Eq_GoodFamily GM1}%
\end{equation}
We can now turn to the%

%TCIMACRO{\TeXButton{VSs}{\vspace{0.3cm}}}%
%BeginExpansion
\vspace{0.3cm}%
%EndExpansion

\begin{proof}
[\textbf{Proof of Theorem \ref{T_LLTGMCyls1}}]\textbf{(i)} It is well known
that in situation A) as well as in situation B) the hitting times of the
$A_{l}$ satisfy
\begin{equation}
\mu(\mu(A_{l})\varphi_{A_{l}}\leq t)\Longrightarrow1-e^{-\theta t}\text{ \quad
as }l\rightarrow\infty\label{Eq_jshfvbjhesbvfjhvbhjb}%
\end{equation}
(see e.g. Theorem 10.1 of \cite{ZWhenAndWhere}). Now fix any $r>0$. Since
$\mathcal{U}(r^{\prime})\subseteq\mathcal{U}(r)$ and $\mathcal{U}%
_{l}(r^{\prime})\subseteq\mathcal{U}_{l}(r)$ for $0<r^{\prime}<r$, it is clear
that we can assume w.l.o.g. that $r\geq r_{\ast}=r_{\ast}(\mathfrak{S})$. We
are going to employ Theorem \ref{T_LLT_ConsecutiveTimes} a) and b).\newline%
\newline\textbf{(ii)} To this end, we first check the assumptions of Theorem
\ref{T_LLT_RetGeneral} with $A_{l}^{\bullet}:=\varnothing$ in case A), and
$A_{l}^{\bullet}:=A_{l}\cap T^{-p}A_{l}=A_{p}\cap T^{-p}A_{l}=A_{l+p}\in
\xi_{l+p}\ $for $l\geq p$ in case B). Beginning with the conditions
(\ref{Eq_ReturningPartGeneral1}) and (\ref{Eq_ReturningPartGeneral2}), we
recall that these are trivially satisfied in case A) (compare Remark
\ref{Rem_Afterthought1}.b). In case B) we find that
\[
\mu(A_{l}^{\bullet})=\mu(A_{p}\cap T^{-p}A_{l})=\mu(A_{l})\int_{A_{l}}%
v_{A_{p}}^{\prime}\,d\mu_{A_{l}}\text{.}%
\]
Now $x^{\ast}\in A_{l}\subseteq A_{1}\subseteq T^{p}A_{p}$, and $\mathrm{diam}%
(A_{l})\searrow0$. As $v_{A_{p}}^{\prime}$ is continuous on $T^{p}A_{p}$ with
$v_{A_{p}}^{\prime}(x^{\ast})=1-\theta$, we get $\mu(A_{l}^{\bullet}%
)\sim(1-\theta)\mu(A_{l})$ and hence%
\begin{equation}
\mu(A_{l}^{\circ})\sim\theta\mu(A_{l})\text{ \quad as }l\rightarrow
\infty\text{.}\label{Eq_Spades}%
\end{equation}
But if $\nu_{l}\in\mathcal{V}_{l}$, then $\nu_{l}=u_{l}\odot\mu$ with
$u_{l}=1_{A_{l}}u_{l}$ satisfying $\mathrm{R}_{A_{l}}(u)\leq r$, so that
$\nu_{l}(A_{l}^{\circ})=\int_{A_{l}^{\circ}}u_{l}\,d\mu/\int_{A_{l}}%
u_{l}\,d\mu$, and (\ref{Eq_RatiosOnShortSets}) ensures that
\begin{equation}
\frac{1}{1+r\,\mathrm{diam}(A_{l})}\,\frac{\mu(A_{l}^{\circ})}{\mu(A_{l})}%
\leq\nu_{l}(A_{l}^{\circ})\leq(1+r\,\mathrm{diam}(A_{l}))\,\frac{\mu
(A_{l}^{\circ})}{\mu(A_{l})}\text{,}\label{Eq_Checkers}%
\end{equation}
with bounds not depending on $\nu_{l}$ and both converging to $\theta$ since
$\mathrm{diam}(A_{l})\rightarrow0$. This implies
(\ref{Eq_ReturningPartGeneral1}). Finally, observe that $\varphi_{A_{l}}\leq
p$ on $A_{l}^{\bullet}$, which entails (\ref{Eq_ReturningPartGeneral2}%
).\newline\newline\textbf{(iii)} To validate condition
(\ref{Eq_RetEquivHitGeneral}) we show that (\ref{Eq_OrderOfTheTauLGeneral1}%
)-(\ref{Eq_NothingHappensTooQuicklyGeneral}) are satisfied. Define $p:=0$ in
case A) and set $\tau_{l}:=l+p$ both in case A) and in case B). Then
(\ref{Eq_GMMapsContractionOfCylMeasure}) immediately shows $\mu(A_{l}%
)\,\tau_{l}\rightarrow0$, so that condition (\ref{Eq_OrderOfTheTauLGeneral1})
is met in either case.

Recall that $\mathcal{V}$ is a LLT-uniform family of measures, see
(\ref{Eq_GoodFamily GM1}). For $\nu_{l}\in\mathcal{V}_{l}$ with $\nu_{l}%
=u_{l}\odot\mu$ where $u_{l}=1_{A_{l}}u_{l}$ satisfies $\mathrm{R}_{A_{l}%
}(u)\leq r$, we have $d\nu_{l}^{\circ}/d\mu=(%
%TCIMACRO{\tint \nolimits_{A_{l}^{\circ}}}%
%BeginExpansion
{\textstyle\int\nolimits_{A_{l}^{\circ}}}
%EndExpansion
u_{l}\,d\mu)^{-1}1_{A_{l}^{\circ}}u_{l}$, and since $A_{l}^{\bullet}\in
\xi_{l+p}$, the set $A_{l}^{\circ}$ is a union of cylinders from $\xi_{l+p}$.
Properties (\ref{Eq_NeoNeo}) and (\ref{Eq_ClosedCtbleConvex}) thus give
\[
\frac{d(\nu_{l}^{\circ}\circ T^{-\tau_{l}})}{d\mu}=\widehat{T}^{l+p}\left(
\frac{1_{A_{l}^{\circ}}\,u_{l}}{%
%TCIMACRO{\tint \nolimits_{A_{l}^{\circ}}}%
%BeginExpansion
{\textstyle\int\nolimits_{A_{l}^{\circ}}}
%EndExpansion
u_{l}\,d\mu}\right)  \in\mathcal{U}(r)\text{.}%
\]
Hence, $\nu_{l}^{\circ}\circ T^{-\tau_{l}}\in\mathcal{V}$, proving
(\ref{Eq_SailingAfterTauGeneral}), as $A_{l}^{\circ}\cap\{\tau_{l}%
=j\}\in\{A_{l}^{\circ},\varnothing\}$ for constant $\tau_{l}$.\newline%
\newline\textbf{(iv)} As a preparation for the argument to follow below, we
next show that%
\begin{equation}
\text{given }j^{\ast}\text{ there is some }l^{\ast}\text{ such that \quad
}\varphi_{A_{l}}\geq j^{\ast}\text{\ on }A_{l}^{\circ}\text{\quad for }l\geq
l^{\ast}\text{.}\label{Eq_ExiLstar}%
\end{equation}
Note first that for any fixed $i\geq1$ with $T^{i}x^{\ast}\neq x^{\ast}$ there
is some $l^{\ast}(i)\geq1$ such that
\[
T^{i}A_{l}\cap A_{l}=\varnothing\text{ for }l\geq l^{\ast}(i)\text{. }%
\]
(Because $T^{i}$ is continuous at $x^{\ast}$ and $\mathrm{diam}(A_{l}%
)\searrow0$.) To prove (\ref{Eq_ExiLstar}) in case A), apply this to each
$i\in\{1,\ldots,j^{\ast}-1\}$ and take $l^{\ast}:=\max(l^{\ast}(1),\ldots
,l^{\ast}(j^{\ast}-1))$.

Establishing (\ref{Eq_ExiLstar}) in case B) requires a little more thought. We
can easily rule out quick returns by the same argument as in case A), since
taking $l^{\prime}:=\max(l^{\ast}(1),\ldots,l^{\ast}(p-1))$ leads to
\begin{equation}
\varphi_{A_{l}}\geq p\text{\quad on }A_{l}\text{\quad for }l\geq l^{\prime
}\text{.}\label{Eq_jewahfbjeahfb}%
\end{equation}
Now consider potential return times $j\geq p$. Note first that $A_{l}^{\circ}$
is $\xi_{l+p}$-measurable with $A_{l}^{\circ}=\bigcup\nolimits_{W\in
\mathcal{W}_{l}}W$ where $\mathcal{W}_{l}:=\{W\in\xi_{l+p}\setminus
\{A_{l+p}\}:W\subseteq A_{l}\}$. We claim that
\begin{equation}
T^{j}W\cap A_{l}=\varnothing\text{ \quad for }W\in\mathcal{W}_{l}\text{ \quad
if }j\leq l-p\text{ and }l\geq2p\text{.}\label{Eq_NowMoreDetails2}%
\end{equation}
Combined with (\ref{Eq_jewahfbjeahfb}) this gives
\begin{equation}
\varphi_{A_{l}}>l-p\text{ \quad on }A_{l}^{\circ}\text{\quad for }l\geq
\max(2p,l^{\prime})\text{,}%
\end{equation}
which in turn implies (\ref{Eq_ExiLstar}) with $l^{\ast}:=\max(2p,l^{\prime
},j^{\ast}+p)$.

We validate (\ref{Eq_NowMoreDetails2}). Call $p$ a period of a finite tuple
$(V_{1},\ldots,V_{k})$ if $V_{i+p}=V_{i}$ whenever $1\leq i\leq k-p$. Suppose
now that $j\geq p$. Writing $Z_{i}:=\xi(T^{i}x^{\ast})\in\xi$ it is clear that
$p$ is a period of $(Z_{i})_{i\geq0}$ and $A_{n}=[Z_{0},\ldots,Z_{n-1}]$ for
$n\geq1$. Take any $W\in\mathcal{W}_{l}$, then $W=[W_{0},\ldots,W_{l+p-1}]$
with $W_{i}\in\xi$ satisfying $(W_{0},\ldots,W_{l-1})=(Z_{0},\ldots,Z_{l-1})$
(since $W\subseteq A_{l}$), while $(W_{l},\ldots,W_{l+p-1})\neq(Z_{l}%
,\ldots,Z_{l+p-1})$ (because $W\neq A_{l+p}$). This implies that $p$ is not a
period of the tuple $(W_{j},\ldots,W_{l+p-1})$ of length $m:=l+p-j\geq2p$
(since otherwise $W_{i}=Z_{i}$ for all $i<l+p$ and hence $W=A_{l+p}$). Note
also that $T^{j}W\subseteq\lbrack W_{j},\ldots,W_{l+p-1}]\in\xi_{m}$. Now
assume, for a contradiction, that $T^{j}W\cap A_{l}\neq\varnothing$. Then
$[W_{j},\ldots,W_{l+p-1}]\cap\lbrack Z_{0},\ldots,Z_{l-1}]\neq\varnothing$,
and since $j\geq p$ entails $l\geq m$, we conclude that $(W_{j},\ldots
,W_{l+p-1})=(Z_{0},\ldots,Z_{m-1})$ which has $p$ as a period, contradicting
the fact stated above. This proves (\ref{Eq_NowMoreDetails2}).\newline%
\newline\textbf{(v)} It remains to check
(\ref{Eq_NothingHappensTooQuicklyGeneral}) for a given sequence $(k_{l}%
)\in\mathcal{K}(A_{l})$. We actually prove a more general statement: Given
$B_{l}\in\mathcal{A}$, $l\geq1$, we will show that
\begin{align}
\mu_{A_{l}}(\{\varphi_{A_{l}} &  \leq\tau_{l}\}\cap T^{-k_{l}}B_{l})=o\left(
\mu(B_{l})\right)  \text{ \quad and}\label{Eq_qpqpqpqp}\\
\mu_{A_{l}^{\circ}}(\{\varphi_{A_{l}} &  \leq\tau_{l}\}\cap T^{-k_{l}}%
B_{l})=o\left(  \mu(B_{l})\right)  \text{ \quad as }l\rightarrow\infty
\text{.}\nonumber
\end{align}
But for $A_{l}^{\prime}\in\mathcal{A}\cap A_{l}$, statements (\ref{Eq_Spades})
and (\ref{Eq_Checkers}) give, for $\nu_{l}=u_{l}\odot\mu$ as before,
\[
\nu_{l}^{\circ}(A_{l}^{\prime})\leq\nu_{l}^{\circ}(A_{l}^{\circ})^{-1}\nu
_{l}(A_{l}^{\prime})\leq2\theta^{-1}(1+r\,\mathrm{diam}(A_{l}))\nu_{l}%
(A_{l}^{\prime})\text{ \quad for }l\geq l_{0}\text{,}\,
\]
and recalling (\ref{Eq_RatiosOnShortSets}) we have $\nu_{l}(A_{l}^{\prime
})=\int_{A_{l}^{\prime}}u_{l}\,d\mu/\int_{A_{l}}u_{l}\,d\mu\leq
(1+r\,\mathrm{diam}(X))\,\mu_{A_{l}}(A_{l}^{\prime})$. Thus,
\begin{equation}
\nu_{l}^{\circ}(A_{l}^{\prime})\leq2\theta^{-1}(1+r\,\mathrm{diam}%
(X))^{2}\,\mu_{A_{l}}(A_{l}^{\prime})\text{ \quad for }l\geq l_{0}\text{,
}A_{l}^{\prime}\in\mathcal{A}\cap A_{l}\text{,}%
\label{Eq_dsjfbkwaebfkjwebfhkjwbfhbwfhbj}%
\end{equation}
and, taking $A_{l}^{\prime}:=A_{l}\cap\{\varphi_{A_{l}}\leq\tau_{l}\}\cap
T^{-k_{l}}B_{l}$ here, (\ref{Eq_qpqpqpqp}) implies
(\ref{Eq_NothingHappensTooQuicklyGeneral}), since it shows that, more
generally,
\begin{equation}
\nu_{l}^{\circ}(\{\varphi_{A_{l}}\leq\tau_{l}\}\cap T^{-k_{l}}B_{l}%
)=o(\mu(B_{l}))\text{ \quad}%
\begin{array}
[c]{c}%
\text{as }l\rightarrow\infty\text{, uniformly}\\
\text{in }(\nu_{l})\in%
%TCIMACRO{\tprod \nolimits_{l\geq1}}%
%BeginExpansion
{\textstyle\prod\nolimits_{l\geq1}}
%EndExpansion
\mathcal{V}_{l}\text{.}%
\end{array}
\label{Eq_qpqpqpqpqpqpqp2}%
\end{equation}
\textbf{(vi)} Consider large problematic values \thinspace$j\in\{l+1,\ldots
,l+p\}$ of $\varphi_{A_{l}}$ first. For such $j$ and indices $l $ so large
that $\mu(A_{l}^{\circ})\geq\theta\mu(A_{l})/2$ and $k_{l}>l+p$ we find,
applying (\ref{Eq_UpperDistortBound}) twice, that
\begin{align*}
\mu_{A_{l}^{\circ}}(\{\varphi_{A_{l}}=j\}\cap T^{-k_{l}}B_{l})\leq & \mu
(A_{l}^{\circ})^{-1}\mu(A_{l}^{\circ}\cap T^{-j}A_{l}\cap T^{-k_{l}}B_{l}))\\
& \leq2\theta^{-1}\mu(A_{l})^{-1}\mu(A_{l}\cap T^{-l}(T^{-j+l}A_{l}\cap
T^{-k_{l}+l}B_{l}))\\
& \leq2\theta^{-1}\flat^{-1}e^{r}\mu(T^{-j+l}A_{l}\cap T^{-k_{l}+l}B_{l})\\
& =2\theta^{-1}\flat^{-1}e^{r}\mu(A_{l}\cap T^{-k_{l}+j}B_{l})\\
& \leq2\theta^{-1}\flat^{-2}e^{2r}\mu(A_{l})\mu(B_{l})\text{.}%
\end{align*}
Consequently,%
\[
\mu_{A_{l}^{\circ}}(\{l<\varphi_{A_{l}}\leq l+p\}\cap T^{-k_{l}}%
B_{l})=o\left(  \mu(B_{l})\right)  \text{ \quad as }l\rightarrow\infty\text{,}%
\]
and by the same argument we obtain the corresponding statement with
$\mu_{A_{l}^{\circ}}$ replaced by $\mu_{A_{l}}$. Therefore (\ref{Eq_qpqpqpqp})
follows once we prove%
\begin{align}
\mu_{A_{l}}(\{\varphi_{A_{l}} &  \leq l\}\cap T^{-k_{l}}B_{l})=o\left(
\mu(B_{l})\right)  \text{ \quad and}\label{Eq_qpqpqpqp2}\\
\mu_{A_{l}^{\circ}}(\{\varphi_{A_{l}} &  \leq l\}\cap T^{-k_{l}}%
B_{l})=o\left(  \mu(B_{l})\right)  \text{ \quad as }l\rightarrow\infty
\text{.}\nonumber
\end{align}
\newline\newline\textbf{(vii)} Note that for integers with $l\geq j\geq1$ and
$k>j+l$ we have $A_{l}\subseteq A_{j}$ and therefore, using
(\ref{Eq_UpperDistortBound}) twice,
\begin{align}
\mu_{A_{l}^{\circ}}(T^{-j}A_{l}\cap T^{-k}B_{l})  & =\mu(A_{l}^{\circ}%
)^{-1}\mu(A_{l}^{\circ}\cap T^{-j}(A_{l}\cap T^{-(k-j)}B_{l}))\nonumber\\
& \leq\mu(A_{l}^{\circ})^{-1}\mu(A_{l}\cap T^{-j}(A_{l}\cap T^{-(k-j)}%
B_{l}))\nonumber\\
& \leq\mu(A_{l}^{\circ})^{-1}\mu(A_{j}\cap T^{-j}(A_{l}\cap T^{-(k-j)}%
B_{l}))\nonumber\\
& \leq\flat^{-1}e^{r}\mu(A_{l}^{\circ})^{-1}\mu(A_{j})\mu(A_{l}\cap
T^{-(k-j)}B_{l})\nonumber\\
& \leq\flat^{-2}e^{2r}\mu(A_{l}^{\circ})^{-1}\mu(A_{j})\mu(A_{l}%
)\mu(T^{-(k-j-l)}B_{l})\nonumber\\
& =\flat^{-2}e^{2r}\mu(A_{l}^{\circ})^{-1}\mu(A_{j})\mu(A_{l})\mu
(B_{l})\text{.}\label{Eq_Clubs}%
\end{align}
The same calculation (with $\mu(A_{l}^{\circ})^{-1}$ replaced by $\mu
(A_{l})^{-1}$) yields
\begin{equation}
\mu_{A_{l}}(T^{-j}A_{l}\cap T^{-k}B_{l})\leq\flat^{-2}e^{2r}\mu(A_{j}%
)\mu(B_{l})\text{.}\label{Eq_Clubs2}%
\end{equation}

Now take any sequence $(k_{l})\in\mathcal{K}(A_{l})$ and pick some
$\varepsilon>0$. Choose $j^{\ast}\geq1$ so large that $2\kappa e^{2r}%
q^{j^{\ast}}<\varepsilon\theta\flat^{2}(1-q)$ and pick $l^{\ast}$ as in
(\ref{Eq_ExiLstar}).

As $(\mu(A_{l})k_{l})$ is bounded, (\ref{Eq_GMMapsContractionOfCylMeasure})
shows that the $k_{l}$ grow exponentially fast, and there is some $l^{\ast
\ast}\geq l^{\ast}$ such that $k_{l}>j^{\ast}+l$ for $l\geq l^{\ast\ast} $. In
view of (\ref{Eq_Spades}) we can also assume that $\mu(A_{l})/\mu(A_{l}%
^{\circ})\leq2/\theta$ for $l\geq l^{\ast\ast}$. Property (\ref{Eq_ExiLstar})
now allows us to employ (\ref{Eq_Clubs}) in the following estimate. For $l\geq
l^{\ast\ast}$,%
\begin{align*}
\mu_{A_{l}^{\circ}}(\{\varphi_{A_{l}}\leq l\}\cap T^{-k_{l}}B_{l})  &
=\mu_{A_{l}^{\circ}}\left(  \left(  \bigcup_{j=j^{\ast}}^{l}T^{-j}%
A_{l}\right)  \cap T^{-k_{l}}B_{l}\right) \\
& \leq\sum_{j=j^{\ast}}^{l}\mu_{A_{l}^{\circ}}\left(  T^{-j}A_{l}\cap
T^{-k_{l}}B_{l}\right) \\
& \leq\frac{2e^{2r}}{\theta\flat^{2}}\,\mu(B_{l})\sum_{j=j^{\ast}}^{l}%
\mu(A_{j})\\
& \leq\frac{2\kappa e^{2r}}{\theta\flat^{2}}\,\mu(B_{l})\sum_{j\geq j^{\ast}%
}q^{j}\\
& <\varepsilon\,\mu(B_{l})\text{,}%
\end{align*}
which proves (\ref{Eq_qpqpqpqp2}) for $\mu_{A_{l}^{\circ}}$, while the
assertion for $\mu_{A_{l}}$ follows in exactly the same way using
(\ref{Eq_Clubs2}) rather than (\ref{Eq_Clubs}).\newline\newline\textbf{(viii)}
Returning to the envisaged application of Theorem \ref{T_LLT_ConsecutiveTimes}
a), we have to validate the extra condition (\ref{Eq_InvarianceVls}). Take
$l\geq1$ and any $\nu_{l}\in\mathcal{V}_{l}$, so that $\nu_{l}$ has density
$u_{l}\in\mathcal{U}_{l}=\{u\in\mathcal{D}(\mu):u=1_{A_{l}}u$ and
$\mathrm{R}_{A_{l}}(u)\leq r\}$. Then,%
\[
w_{l}:=\frac{d(\nu_{l,\{\varphi_{A_{l}}=k_{l}\}}\circ T_{A_{l}}^{-1})}{d\mu
}=\frac{\widehat{T}^{k_{l}}\left(  1_{A_{l}\cap\{\varphi_{A_{l}}=k_{l}\}}%
u_{l}\right)  }{\int_{A_{l}\cap\{\varphi_{A_{l}}=k_{l}\}}u_{l}\,d\mu}\text{,}%
\]
which is obviously supported on $A_{l}$, so that $w_{l}=1_{A_{l}}w_{l}$. To
see that $w_{l}\in\mathcal{U}_{l}$ and hence $\nu_{l,\{\varphi_{A_{l}}%
=k_{l}\}}\circ T_{A_{l}}^{-1}\in\mathcal{V}_{l}$, it remains to show that
$\mathrm{R}_{A_{l}}(w_{l})\leq r$. Since $A_{l}\in\xi_{l}$, we see that the
set $\{\varphi_{A_{l}}=k_{l}\}$, and hence also $A_{l}\cap\{\varphi_{A_{l}%
}=k_{l}\}$, is $\xi_{k_{l}+l}$-measurable,
\[
A_{l}\cap\{\varphi_{A_{l}}=k_{l}\}=\bigcup_{W\in\mathcal{W}_{l}}%
W\text{\quad(disjoint),}%
\]
where $\mathcal{W}_{l}:=\{W\in\xi_{k_{l}+l}:W\subseteq A_{l}\cap
\{\varphi_{A_{l}}=k_{l}\}\}$. Hence,
\[
w_{l}=\sum_{W\in\mathcal{W}_{l}}\frac{\int_{W}u_{l}\,d\mu}{\int_{A_{l}%
\cap\{\varphi_{A_{l}}=k_{l}\}}u_{l}\,d\mu}\cdot\,\widehat{T}^{k_{l}}\left(
\frac{1_{W}\,u_{l}}{\int_{W}u_{l}\,d\mu}\right)  \text{.}%
\]
Since by (\ref{Eq_NeoNeoNeo}) each $\mathcal{U}_{l}$ is closed under countable
convex combinations, we need only show that $\,\widehat{T}^{k_{l}}%
(1_{W}\,u_{l}/\int_{W}u_{l}\,d\mu)\in\mathcal{U}_{l}$ for every $W\in
\mathcal{W}_{l}$. As $\mathrm{R}_{A_{l}}$ is invariant under scaling, the
latter follows once we validate
\begin{equation}
\mathrm{R}_{A_{l}}\left(  \widehat{T}^{k_{l}}(1_{W}\,u_{l})\right)  \leq
r\text{ \qquad for }W\in\mathcal{W}_{l}\text{.}%
\label{Eq_kefjuvfkebfienjkfijnv}%
\end{equation}
Now $W\subseteq A_{l}$ and in view of (\ref{Eq_MbilityOfImages}), $T^{k_{l}%
}W\subseteq A_{l}$ is $\xi_{l}$-measurable and hence coincides with $A_{l}$
(mod $\mu$). But since $\mathrm{R}_{A_{l}}(u)\leq r$, the observation
(\ref{Eq_NeoNeo}) entails (\ref{Eq_kefjuvfkebfienjkfijnv}), thus completing
the proof of (\ref{Eq_InvarianceVls}).

To apply b) of Theorem \ref{T_LLT_ConsecutiveTimes}, it remains to check
(\ref{Eq_Hourglass0}), that is, $\nu_{\{\varphi_{A_{l}}=k_{l}\}}\circ
T_{A_{l}}^{-1}\in\mathcal{V}_{l}$ for $l\geq1$ and $\nu\in\mathcal{V}$. This
follows by a variant of the argument used to prove (\ref{Eq_InvarianceVls}).
\end{proof}%

%TCIMACRO{\TeXButton{VSs}{\vspace{0.3cm}}}%
%BeginExpansion
\vspace{0.3cm}%
%EndExpansion

The argument for unions of rank-one cylinders is similar:

\begin{proof}
[\textbf{Proof of Theorem \ref{T_LLTGMCyls2}}]\textbf{(i)} Again it is
standard that the hitting time distributions of the $A_{l}$ converge to a
normalized exponential law, $\mu(\mu(A_{l})\varphi_{A_{l}}\leq
t)\Longrightarrow1-e^{-t}$ as $l\rightarrow\infty$ (see e.g. Theorem 10.2 of
\cite{ZWhenAndWhere}). Fix any $r>0$, w.l.o.g. with $r\geq r_{\ast}=r_{\ast
}(\mathfrak{S})$. We are going to apply Theorem \ref{T_LLT_ConsecutiveTimes}
again.\newline\newline\textbf{(ii)} We begin by checking the asumptions of
Theorem \ref{T_LLT_RetGeneral} with $A_{l}^{\bullet}:=\varnothing$,
$A_{l}^{\circ}:=A_{l}$, $\theta:=1$. Assumptions
(\ref{Eq_ReturningPartGeneral1}) and (\ref{Eq_ReturningPartGeneral2}) are
trivially satisfied. The remaining condition (\ref{Eq_RetEquivHitGeneral})
will be shown to hold via (\ref{Eq_OrderOfTheTauLGeneral1}%
)-(\ref{Eq_NothingHappensTooQuicklyGeneral}). Here we take $\tau_{l}:=1$ for
all $l$, so that (\ref{Eq_OrderOfTheTauLGeneral1}) is certainly fulfilled.
Note that as a consequence of (\ref{Eq_GoodFamily GM1}), $\mathcal{V}$ is a
LLT-uniform family for $(A_{l})$: We have $\widehat{T}v\in\mathcal{U}(r)$ for
$v\in\mathcal{V}$, because of the $n=1$ case of (\ref{Eq_ImageDensiCylGM}) and
(\ref{Eq_ClosedCtbleConvex}). As $\mathcal{V}_{l}\subseteq\mathcal{V}$, this
argument also proves (\ref{Eq_SailingAfterTauGeneral}).\newline\newline%
\textbf{(iii)} We still need to validate
(\ref{Eq_NothingHappensTooQuicklyGeneral}), and to this end fix some
$(k_{l})\in\mathcal{K}(A_{l})$ and show first that for any $B_{l}%
\in\mathcal{A}$, and any $\xi$-measurable $C_{l}\subseteq A_{l}$ we have
\begin{equation}
\mu_{C_{l}}(\{\varphi_{A_{l}}\leq1\}\cap T^{-k_{l}}B_{l})=o\left(  \mu
(B_{l})\right)  \text{ \quad as }l\rightarrow\infty\text{.}%
\label{Eq_cvcvcvcvcvvbbdbbsbd}%
\end{equation}
Indeed, taking $l$ so large that $k_{l}>2$ and employing
(\ref{Eq_UpperDistortBound}) twice, we see via $T$-invariance of $\mu$ that%
\begin{align*}
\mu(C_{l}\cap\{\varphi_{A_{l}}\leq1\}\cap T^{-k_{l}}B_{l})=  & \sum_{Z\in
\xi:Z\subseteq C_{l}}\mu(Z\cap T^{-1}(A_{l}\cap T^{-(k_{l}-1)}B_{l}))\\
& \leq\flat^{-1}e^{r}\sum_{Z\in\xi:Z\subseteq C_{l}}\mu(Z)\,\mu(A_{l}\cap
T^{-(k_{l}-1)}B_{l}))\\
& =\flat^{-1}e^{r}\mu(C_{l})\,\mu(A_{l}\cap T^{-1}(T^{-(k_{l}-2)}B_{l}))\\
& \leq\flat^{-2}e^{2r}\mu(C_{l})\,\mu(A_{l})\,\mu(B_{l})\text{,}%
\end{align*}
which immediately yields (\ref{Eq_cvcvcvcvcvvbbdbbsbd}). Then, for $\nu_{l}%
\in\mathcal{V}_{l}$ with density $(\int_{C_{l}}u_{l}\,d\mu)^{-1}1_{C_{l}%
}\,u_{l}$, where $u_{l}\in\mathcal{U}(r)$ and $C_{l}\subseteq A_{l}$ is a
$\xi$-measurable set we obtain, using (\ref{Eq_RRRegularity}) and the
preceding estimate (with $B\cap C_{l}$ in place of $C_{l}$),
\begin{align*}
\nu_{l}(\{\varphi_{A_{l}} &  \leq1\}\cap T^{-k_{l}}B_{l})=\frac{\int
_{C_{l}\cap\{\varphi_{A_{l}}\leq1\}\cap T^{-k_{l}}B_{l}}u_{l}\,d\mu}%
{\int_{C_{l}}u_{l}\,d\mu}\\
&  =\sum_{B\in\beta:\int_{B\cap C_{l}}u_{l}\,d\mu>0}\frac{\int_{B\cap
C_{l}\cap\{\varphi_{A_{l}}\leq1\}\cap T^{-k_{l}}B_{l}}u_{l}\,d\mu}{\int_{B\cap
C_{l}}u_{l}\,d\mu}\,\frac{\int_{B\cap C_{l}}u_{l}\,d\mu}{\int_{C_{l}}%
u_{l}\,d\mu}\\
&  \leq\sum_{B\in\beta:\int_{B\cap C_{l}}u_{l}\,d\mu>0}\frac{\sup_{B}u_{l}%
}{\inf_{B}u_{l}}\,\frac{\mu(B\cap C_{l}\cap\{\varphi_{A_{l}}\leq1\}\cap
T^{-k_{l}}B_{l})}{\mu(B\cap C_{l})}\,\frac{\int_{B\cap C_{l}}u_{l}\,d\mu}%
{\int_{C_{l}}u_{l}\,d\mu}\\
&  \leq(1+r\,\mathrm{diam}(X))\cdot\flat^{-2}e^{2r}\,\mu(A_{l})\,\mu
(B_{l})\cdot\sum_{B\in\beta:\int_{B\cap C_{l}}u_{l}\,d\mu>0}\frac{\int_{B\cap
C_{l}}u_{l}\,d\mu}{\int_{C_{l}}u_{l}\,d\mu}\text{,}%
\end{align*}
where the rightmost sum equals $1$. Taking $B_{l}:=A_{l}$ yields
(\ref{Eq_NothingHappensTooQuicklyGeneral}).\newline\newline\textbf{(iv)} Now
consider the extra conditions (\ref{Eq_InvarianceVls})-(\ref{Eq_Hourglass1})
of Theorem \ref{T_LLT_ConsecutiveTimes}, where we define $c_{l}^{(j)}%
:=\mu_{A_{l}}(C_{l}^{(j)})$.

To take care of (\ref{Eq_InvarianceVls}), pick $\nu_{l}$ with density
$(\int_{C_{l}}u_{l}\,d\mu)^{-1}1_{C_{l}}\,u_{l}$ where $u_{l}\in
\mathcal{U}(r)$ and $C_{l}\subseteq A_{l}$ is some $\xi$-measurable set. Then
$\{\varphi_{A_{l}}\geq k_{l}\}$ is $\xi_{k_{l}}$-measurable, being the union
of all $[Z_{0},\ldots,Z_{k_{l}-1}]$ with $Z_{i}\in\xi$ satisfying $Z_{i}\cap
A_{l}=\varnothing$. For $Z\in\xi_{k_{l}}$ we thus have $Z\cap\{\varphi_{A_{l}%
}=k_{l}\}=Z\cap T^{-k_{l}}A_{l}$ if $Z\subseteq\{\varphi_{A_{l}}\geq k_{l}\}$
while $Z\cap\{\varphi_{A_{l}}=k_{l}\}=\varnothing$ otherwise. Therefore,%
\[
C_{l}\cap\{\varphi_{A_{l}}=k_{l}\}=%
%TCIMACRO{\dbigcup \limits_{Z\in\xi_{k_{l}},Z\subseteq C_{l}\cap\{\varphi
%_{A_{l}}\geq k_{l}\}}}%
%BeginExpansion
{\displaystyle\bigcup\limits_{Z\in\xi_{k_{l}},Z\subseteq C_{l}\cap
\{\varphi_{A_{l}}\geq k_{l}\}}}
%EndExpansion
Z\cap T^{-k_{l}}A_{l}\text{ \quad(disjoint),}%
\]
so that the left-hand side of (\ref{Eq_InvarianceVls}) has density
$d(\nu_{l,\{\varphi_{A_{l}}=k_{l}\}}\circ T_{A_{l}}^{-1})/d\mu$ given by%
\begin{equation}
1_{A_{l}}\sum_{Z\in\xi_{k_{l}},Z\subseteq C_{l}\cap\{\varphi_{A_{l}}\geq
k_{l}\}}\frac{\int_{Z}u_{l}\,d\mu}{\int_{C_{l}\cap\{\varphi_{A_{l}}\geq
k_{l}\}}u_{l}\,d\mu}\,\widehat{T}^{k_{l}}\left(  \frac{1_{Z}\,u_{l}}{\int
_{Z}u_{l}\,d\mu}\right)  \text{,}\label{Eq_TdS}%
\end{equation}
which due to (\ref{Eq_ImageDensiCylGM}) and (\ref{Eq_ClosedCtbleConvex}) shows
that $\nu_{l,\{\varphi_{A_{l}}=k_{l}\}}\circ T_{A_{l}}^{-1}$ belongs to
$\mathcal{V}_{l}$. Condition (\ref{Eq_Hourglass0}) follows by the same type of argument.

Validity of (\ref{Eq_kysdjhfbhjeabgfhjbevb}) is immediate from the definition
of $\mathcal{V}_{l}$ since the $C_{l}^{(j)}$ are $\xi$-measurable. Regarding
(\ref{Eq_AsyNuCl}) note that $\nu_{l,\{\varphi_{A_{l}}=k_{l}\}}\circ T_{A_{l}%
}^{-1}$ belongs to $\mathcal{V}_{l}$ and by (\ref{Eq_TdS}) has a density of
the form $(%
%TCIMACRO{\tint \nolimits_{A_{l}}}%
%BeginExpansion
{\textstyle\int\nolimits_{A_{l}}}
%EndExpansion
\widetilde{u}_{l}\,d\mu)^{-1}1_{A_{l}}\widetilde{u}_{l}$ (for a suitable
$\widetilde{u}_{l}\in\mathcal{U}(r)$), so that by (\ref{Eq_RatiosOnShortSets})
and the assumption that $A_{l}\subseteq A_{l}^{\prime}\in\beta$,
\[
\nu_{l,\{\varphi_{A_{l}}=k_{l}\}}\circ T_{A_{l}}^{-1}(C_{l}^{(j)})=\frac
{\int_{C_{l}^{(j)}}\widetilde{u}_{l}\,d\mu}{\int_{A_{l}}\widetilde{u}%
_{l}\,d\mu}=(1+r\,\mathrm{diam}(A_{l}))^{\pm2}\,\mu_{A_{l}}(C_{l}%
^{(j)})\text{,}%
\]
and (\ref{Eq_AsyNuCl}) follows as $\mathrm{diam}(A_{l})\rightarrow0$. Note
that (\ref{Eq_Hourglass1}) follows analogously, because $\nu_{\{\varphi
_{A_{l}}=k_{l}\}}\circ T_{A_{l}}^{-1}$, too, has a density of said type.
\end{proof}%

%TCIMACRO{\TeXButton{VSs}{\vspace{0.3cm}}}%
%BeginExpansion
\vspace{0.3cm}%
%EndExpansion

To conclude this section, we provide a

\begin{proof}
[\textbf{Proof of the claims made in Example \ref{Ex_CF}}]To derive
(\ref{Eq_CF_LLT1}) from Theorem \ref{T_LLTGMCyls2} B), consider the $\xi
$-measurable sets $A_{l}:=\{\mathsf{a}\geq l\}=%
%TCIMACRO{\tbigcup \nolimits_{k\geq l}}%
%BeginExpansion
{\textstyle\bigcup\nolimits_{k\geq l}}
%EndExpansion
I_{k}$ which satisfy $\mu(A_{l})\sim1/(l\log2)$ and $\mathrm{diam}%
(A_{l})\rightarrow0$ as $l\rightarrow\infty$. Due to stationarity we can
replace the $\tau_{l}^{(j)}$ by the variables $\varphi_{A_{l}}\circ T_{A_{l}%
}^{j-1}$. Take $C_{l}^{(j)}:=I_{a_{l}^{(j)}}\in\xi$, then $\mu(C_{l}%
^{(j)})\sim1/((a_{l}^{(j)})^{2}\,\log2)$. Therefore ($\lozenge_{d}$) with
$c_{l}^{(j)}=\mu_{A_{l}}(C_{l}^{(j)})$ gives (\ref{Eq_CF_LLT1}).

Statement (\ref{Eq_CF_LLTprime}) follows in the same way if we use the sets
$A_{l}^{\prime}:=%
%TCIMACRO{\tbigcup \nolimits_{k\geq l:k\text{ prime}}}%
%BeginExpansion
{\textstyle\bigcup\nolimits_{k\geq l:k\text{ prime}}}
%EndExpansion
I_{k}$ instead of $A_{l}$, which satisfy $\mu(A_{l}^{\prime})\sim1/(l\,\log
l\,\log2)$ as $l\rightarrow\infty$ (compare \cite{SZ}).
\end{proof}%

%TCIMACRO{\TeXButton{VSs}{\vspace{0.3cm}}}%
%BeginExpansion
\vspace{0.3cm}%
%EndExpansion
%

%TCIMACRO{\TeXButton{VSs}{\vspace{0.3cm}}}%
%BeginExpansion
\vspace{0.3cm}%
%EndExpansion

\end{document}